\documentclass[onefignum,onetabnum]{siamart190516}
\usepackage{color}

\usepackage{lipsum}
\usepackage{amsfonts}
\usepackage{graphicx}
\usepackage{epstopdf}
\usepackage{algorithmicx}
\usepackage{algpseudocode}
\ifpdf
  \DeclareGraphicsExtensions{.eps,.pdf,.png,.jpg}
\else
  \DeclareGraphicsExtensions{.eps}
\fi

\newsiamremark{remark}{Remark}
\newsiamremark{example}{Example}
\newsiamremark{hypothesis}{Hypothesis}
\crefname{hypothesis}{Hypothesis}{Hypotheses}
\newsiamthm{claim}{Claim}

\headers{An Efficient Iteration for Extremal Solutions of DARE}{C.-Y. Chiang and H.-Y. Fan}

\title{An efficient iteration for the extremal solutions of discrete-time algebraic Riccati equations}

\author{Chun-Yueh Chiang\thanks{Center for General Education, National Formosa University, Huwei 632, Taiwan, R.O.C.
  (\email{chiang@nfu.edu.tw}).} 
\and Hung-Yuan Fan\thanks{Department of Mathematics, National Taiwan Normal University, Taipei 116325, Taiwan, R.O.C.
  (\email{hyfan@ntnu.edu.tw}).}
}

\usepackage{amsopn}

\ifpdf
\hypersetup{
  pdftitle={An Efficient Iteration for Extremal Solutions of DARE},
  pdfauthor={C.-Y. Chiang and H.-Y. Fan}
}
\fi

\begin{document}

\maketitle

\begin{abstract}
Algebraic Riccati equations (AREs) have been extensively applicable in linear optimal control problems and many efficient numerical methods were developed.
The most attention of numerical solutions is the (almost) stabilizing solution in the past works. Nevertheless, it is an interesting and challenging issue
in finding the extremal solutions of AREs which play a vital role in the applications. In this paper, based on the semigroup property,
an accelerated fixed-point iteration (AFPI) is developed for solving the extremal solutions of the discrete-time algebraic Riccati equation.
In addition, we prove that the convergence of the AFPI is at least R-suplinear with order $r>1$ under some mild assumptions.  Numerical examples are shown to illustrate the feasibility and efficiency of the proposed algorithm.
\end{abstract}

\begin{keywords}
  discrete-time algebraic Riccati equation, extremal solution, stabilizing solution, antistabilizing solution,
  accelerated fixed-point iteration, semigroup property
\end{keywords}

\begin{AMS}
  39B12, 39B42, 65H05, 15A24
\end{AMS}

\section{Introduction}
In this paper we are mainly concerned with the extremal solutions of the
discrete-time algebraic Riccati equation (DARE)
\begin{subequations}\label{dare}
\begin{align}
  X = A^H XA - A^H XB(R+B^H XB)^{-1} B^H XA + C^H C, \label{dare-a}
\end{align}
{\color{black} or its equivalent expression}
 \begin{align}
 X = A^H X(I + GX)^{-1}A + H, \label{dare-b}
\end{align}
\end{subequations}
where $A\in \mathbb{C}^{n\times n}$, $B\in \mathbb{C}^{n\times m}$, $R\in \mathbb{C}^{m\times m}$ is positive definite,
$C\in \mathbb{C}^{l\times n}$ with $m, l \leq n$, $I$ is the identity matrix of compatible size,
$G:= BR^{-1}B^H$ and $H=C^H C$, respectively, {\color{black}and the $n$-square matrix $X$ is the unknown Hermitian matrix that is to be determined. For the sake of simplicity, the matrix operator $\mathcal{R}:\mbox{dom}(\mathcal{R})\rightarrow \mathbb{H}_n$ is defined by $\mathcal{R}(X):=A^H X(I + GX)^{-1}A + H$ , which is used to rewrite the equation \cref{dare-b} into the compact expression $X=\mathcal{R}(X)$, where $\mathbb{H}_n$ is the set of all $n\times n$ Hermitian matrices and $\mbox{dom}(\mathcal{R}):=\{X\in\mathbb{H}_n\, |\, \det(R+B^HXB)\neq 0\}$. The following sets
\begin{align*}
& \mathcal{R}_= := \{X\in \mbox{dom}(\mathcal{R})\, |\, X = \mathcal{R}(X) \},\quad  \mathcal{R}_\geq := \{X\in \mbox{dom}(\mathcal{R})\, |\, X \geq \mathcal{R}(X) \}, \label{Rgeq} 
\end{align*}
will play an important role in our main results given below. }

The nonlinear matrix equation of the form \cref{dare-a} arises from the linear-quadratic (LQ) optimal control problem that minimizes the cost functional
\[  \mathcal{J}(u) := \sum_{k=0}^\infty (y_k^H y_k + u_k^H R u_k)  \]
subject to the linear discrete-time system
\[ x_{k+1} = Ax_k + Bu_k,\quad y_k = Cx_k,\quad k\geq 0, \]
 where $x_k\in \mathbb{C}^n$ is the state variable, $u_k\in \mathbb{C}^m$ is the control variable and $y_k$ is the output variable, respectively. {\color{black} Let $F_X = (R + B^HXB)^{-1}B^H XA$ and $T_X = A-BF_X $ for any Hermitian solution $X$ of \cref{dare}. Note that $T_X\equiv (I+GX)^{-1}A$.} If $(A,B)$ is stabilizable and $(A,C)$ is detectable, {\color{black} it is well-known that there exists an optimal state feedback control
 \begin{equation*} 
 u_k^* := -F_{X_{\rm s}} x_k,\quad k \geq 0
 \end{equation*}
 such that the cost functional $\mathcal{J}(u_k^*) = x_0^H {X_{\rm s}} x_0$ is minimized
 and the corresponding closed-loop system
 \begin{equation*} 
 x_{k+1} := T_{X_{\rm s}} x_k,\quad k\geq 0
 \end{equation*}
 is asymptotically stable, i.e., all eigenvalues of the close-loop matrix $T_{X_{\rm s}}$ are inside the open unit disk
 $\mathbb{D}$ in the complex plane, where ${X_{\rm s}}\geq 0$ is the unique stabilizing solution of the DARE \cref{dare-a},}
 see, e.g., \cite{a85,k.s72}. Here, we denote $M \geq 0$ (resp.\ $M > 0$) if $M\in \mathbb{C}^{n\times n}$
 is a Hermitian positive semidefinite (resp.\ positive definite) matrix.
 Analogously, we denote $M \leq 0$ (resp.\ $M < 0$) if $M$  is a Hermitian negative semidefinite (resp.\ negative definite) matrix. A matrix operator $f:\mathbb{H}_{n}\rightarrow\mathbb{H}_{n}$ is order preserving (resp. reversing)  on $\mathbb{H}_{n}$ if $f(A)\geq f(B)$ (resp.  $f(A)\leq f(B)$) when $A\geq B$ and $A,B\in\mathbb{H}_{n}$.
 It is proved that the operator $\mathcal{R}(\cdot)$ is order preserving.
 Throughout the paper 
 the set of all $n\times n$ Hermitian positive semidefinite (resp.\  negative semidefinite) matrices
 is denoted by $\mathbb{N}_n$ (resp.\ $-\mathbb{N}_n$).
For any $M,N\in \mathbb{H}_n$, we write $M \geq N$ (resp.\ $M \leq N$) if $M - N \geq 0$ (resp.\ $M-N \leq 0$).
Moreover, $\sigma (M)$ and $\rho (M)$ are the spectrum and the spectral radius
of the square matrix $M$, {\color{black}respectively}, and define the quantities
\[ \mu(M) := \min \{|\lambda|\ |\ \lambda\in \sigma (M)\},\quad
\rho_{\mathbb{D}} (M) := \max \{|\lambda|\ |\ \lambda\in \sigma (M)\cap \mathbb{D} \}, \]
which will be used below.
We also denote the closed unit disk by $\bar{\mathbb{D}}$, the boundary of $\mathbb{D}$ by $\partial \mathbb{D}$,
the region outside the open unit disk by $\mathbb{D}^c$, and the spectral or Euclidean norm for matrices and vectors in the context by $\|\cdot \|$ , respectively.

If $(A,B)$ is stabilizable and $(A,C)$ is detectable, which are standard assumptions in the LQ optimal control problem mentioned previously,
then $\mathcal{R}_\geq \cap \mathbb{N}_n = \{X_{\rm s}\}$.
In this case, there are dozens of numerical methods for solving the DAREs with small to medium sizes
 in the literature, see, e.g., \cite{c.f.l.w04,g.k.l92,k88,l.r95,l79,p.l.s80,v81} and  the references therein.
 The standard way for computing the stabilizing solution is to utilize the Schur method \cite{g.k.l92, l79}, which solves the
 stable deflating subspace $\mathcal{V} = \mathrm{span}\big(\begin{bmatrix}
   X_1\\ X_2
 \end{bmatrix}\big)$ of the associated symplectic pencil
 \begin{equation*}
 M - \lambda L := \begin{bmatrix}
   A & 0\\ -H & I
 \end{bmatrix} - \lambda \begin{bmatrix}
   I & G\\ 0 & A^H  \end{bmatrix}
   \end{equation*}
using the ordered QZ algorithm, and thus $X=X_2 X_1^{-1}\geq 0$ is the unique stabilizing solution of the DARE \cref{dare-b}.

On the other hand, regarding the DARE \cref{dare-a} or \cref{dare-b} as a nonlinear matrix equation, it is natural to apply Newton's method
for finding the unique stabilizing solution $X_{\rm s}$ under the stabilizability and detectability conditions, see, e.g., \cite{l.r95,m91}.
Moreover, since $H = C^H C\geq 0$, it follows from Theorem 1.1 of \cite{g98} that $X_{\rm s}$ is also the maximal solution
to the DARE \cref{dare-b} satisfying $X_{\rm s}\geq S$ for all $S\in \mathcal{R}_=$, see also \cite[Theorem 13.1.1]{l.r95}.
In addition, stating with some initial $X_0\in \mathbb{H}_n$, the author proposed the theoretical characterizations
for the linearly convergent behavior of Newton's method when solving the DARE \cref{dare-b}
with all unimodular eigenvalues of $T_X$ being semisimple.
Inspired by the doubling algorithm \cite{a77,k88}, starting from the coefficient matrices $A$, $G$ and $H$,
Chu {\em et.\ al.} proposed structure-preserving doubling algorithms (SDAs) \cite{c.f.l.w04} for solving the DARE \cref{dare-b}
and periodic DAREs, respectively. When the control weighting matrix $R\in \mathbb{H}_n$ is singular, we developed a variant of the structured doubling algorithm
for computing the stabilizing solution to the DARE \cref{dare-a}, please refer to \cite{c.f.l10}.

It is well known that the set $\mathcal{R}_\geq \cap \mathbb{N}_n$ associated with the DARE \cref{dare-b}, equipped with the partial ordering ``$\geq$'', forms
a complete lattice if $(A,B)$ is stabilizable and $(A,C)$ is not detectable, see, e.g., \cite{k72,w96}.
Under the same assumptions the DARE \cref{dare-b} might have the almost stabilizing solution $X_s\geq 0$ with $\rho (T_{X_s}) \leq 1$, which is the maximal element of the solution set $\mathcal{R}_=$,
 and the optimizing solution $Y\geq 0$ with the cost functional $\mathcal{J}(u_k^*) = x_0^H Y x_0$ being minimized, which is the minimal element of the solution set
 $\mathcal{R}_\geq \cap \mathbb{N}_n$. On the other hand, it follows from \cite{j81,w95} that the DARE \cref{dare-b} has a unique almost antistabilizing solution
 $\widehat{X}\in -\mathbb{N}_n$ with $\mu (T_{\widehat{X}})\geq 1$, which is the minimal element  of the solution set $\mathcal{R}_=$ satisfying
 $\widehat{X}\leq S$ for all $S\in \mathcal{R}_=$, if $A$ is nonsingular and  $\mathrm{rank} \left[ A-\lambda I, B\right] = n$ for all $\lambda \in \bar{\mathbb{D}}\backslash\{0\}$. Under the same assumptions, it is shown in Theorem 3.1 of \cite{z.c.l17} that
 $\widehat{Y} = -\widetilde{Y}$ is the maximal element of the solution set $\mathcal{R}_=\cap -\mathbb{N}_n$ if and only if
$\widetilde{Y}$ is the minimal positive semidefinite solution to the dual DARE of first kind
\begin{equation}\label{ddare}
  {X} = \widetilde{A}^H \widetilde{X} \widetilde{A} - (\widetilde{C}^H + \widetilde{A}^H {X}\widetilde{B})
(\widetilde{R} + \widetilde{B}^H {X} \widetilde{B})^{-1}(\widetilde{C} + \widetilde{B}^H {X} \widetilde{A}) + \widetilde{H},
\end{equation}
where the coefficient matrices are geiven by
\begin{equation}\label{tldA}
\widetilde{A} = A^{-1},\quad \widetilde{B} = A^{-1}B,\quad \widetilde{H} = (A^{-1})^H H A^{-1},\quad \widetilde{C} = B^H \widetilde{H},\quad
\widetilde{R} = R + \widetilde{C}B,
\end{equation}
respectively. The Riccati equation \cref{ddare} is also called the {\em reverse discrete-time algebraic Riccati equation} presented in \cite{i96},
and a numerically stable algorithm was proposed in \cite{o96} for computing the minimal and antistabilizing solution $\widehat{X}\in -\mathbb{N}_n$ to
the DARE \cref{dare-a} with nonsingular $A$.

Recently, it is shown in Theorem 5.1 of \cite{z.c.l17} that the existence and uniqueness of the minimal solution to the DARE \cref{dare-a}
can be constructed iteratively, through a Newton-type iteration for computing the maximal solution to the dual DARE \cref{ddare},
when $\mathrm{rank} \left[ A-\lambda I, B\right] = n$ for all $\lambda \in \bar{\mathbb{D}}\backslash\{0\}$ and $A$ is nonsingular.
In this paper, applying the methodology of the fixed-point iteration (FPI), we we will give constructive proofs for the existence (or uniqueness) theorems of
these four extremal solutions to the DARE under some wild and reasonable assumptions.
More restrictive assumptions have been addressed in \cite{e.r02} for the uniqueness of the maximal solution to the DARE \cref{dare-b}under the framework
of the FPI.
Starting with some appropriate initial matrices, we will develop the accelerated fixed-point iteration (AFPI), based on the semigroup property of the DARE
with an integer parameter $r\geq 2$ \cite{l.c20}, for computing the extermal elements $(X,Y)$ of the set $\mathcal{R}_\geq \cap \mathbb{N}_n$ simultaneously.
 Specifically, the AFPI with the initial $X_0 = 0$ is just the SDA presented in \cite{c.f.l.w04} when $r=2$.
 In analogous way, a variant of the AFPI will be derived from the dual DARE \cref{ddare}
 for solving the extremal elements $(\widehat{X}, \widehat{Y})$ of the set $\mathcal{R}_=\cap -\mathbb{N}_n$ simultaneously,
 if $(A,B)$ is antistabilizable and $A$ is nonsingular.

The paper is organized as follows. In Section 2 we present some preliminary lemmas
that will be used in our main theorems. Under some wild assumptions, the existence (or uniqueness) of the extremal solutions
to the DARE will be constructed iteratively via the standard FPI in Section 3. Moreover,
based on the framework of the semigroup property for the DARE \cref{dare-b}, two accelerated variants of the FPI will
be proposed for computing the extremal elements of $\mathcal{R}_=\cap\mathbb{N}$ and $\mathcal{R}_=\cap -\mathbb{N}_n$, respectively,
and will discuss the rate of convergence for the AFPI presented in this section.
In Section 4 some illustrate examples are presented for demonstrating the feasibility and efficiency
of the proposed AFPI($r$) with different values of $r$. Finally, we conclude this paper in Section 5.

\section{Preliminaries} \label{sec2}
In this section we introduce some definitions and auxiliary results that will be used below. Firstly, the concepts of R-linear and R-superlinear convergence of a sequence will be used
in the following sections, see, e.g., Definition 1.1 of \cite{l.c20}.
\begin{definition} \label{def-Rconv}
Let $\{X_k\}_{k=0}^\infty$ be a sequence of $n\times n$ complex matrices and $\|\cdot\|$ be an induced matrix norm.
\begin{enumerate}
  \item[(a)] $\{X_k\}_{k=0}^\infty$ converges at least R-linearly to $X_*\in \mathbb{C}^{n\times n}$, if there exists a scalar $\sigma \in (0,1)$ such that
  \[  \limsup_{k\rightarrow \infty} \sqrt[k]{\|X_k - X_*\|} \leq \sigma. \]
\item[(b)] $\{X_k\}_{k=0}^\infty$ converges at least R-superlinearly to $X_*\in \mathbb{C}^{n\times n}$ with order $r>1$,
if there exists a scalar $\sigma \in (0,1)$ such that
  \[  \limsup_{k\rightarrow \infty} \sqrt[r^k]{\|X_k - X_*\|} \leq \sigma. \]
\end{enumerate}
\end{definition}

Let the Stein operator $\mathcal{S}_A:\mathbb{H}_n\rightarrow\mathbb{H}_n$ associated with a matrix $A\in\mathbb{C}^{n\times n}$ be defined by
\begin{align}\label{SME}
\mathcal{S}_A(X) := X-A^H X A,
\end{align}
for all $X\in\mathbb{H}_n$. In general, the operator $\mathcal{S}_A$ is neither order-preserving nor order-reversing.
However, under the assumption that $\rho(A)<1$,
the inverse operator $\mathcal{S}^{-1}_A$ exists and it is order-preserving, since
$\mathcal{S}_A^{-1}(X)=\sum\limits_{k=0}^\infty (A^k)^H X A^k \geq \sum\limits_{k=0}^\infty (A^k)^H Y A^k=\mathcal{S}_A^{-1}(Y)$
for all $X\geq Y$. For the sake of clarity the results in this section can be separated into three major parts.
\subsection{Some useful identities}
The following lemma provides some identities with respect to the Stein matrix operator defined by \cref{SME},
which will play an important role in our main results below.
\begin{lemma} \label{lem2p1}
{\color{black}
Let $X,\widehat{X}\in\mbox{dom}(\mathcal{R})$. The following identities hold, which state the relationship between DARE and a specific Stein matrix equation:
}
\begin{subequations}\label{Req}
\begin{enumerate}
  \item[(i)] If $A_{F}:=
  A-BF$ for any $F\in\mathbb{C}^{m\times n}$ and $H_F:=H+F^H R F$, then
  \begin{align}
  X-\mathcal{R}(X)&=\mathcal{S}_{A_{F}}(X)-H_F+K_F(X), \label{Req-a}
  \end{align}
  where $K_F(X) := (F-F_X)^H(R+B^H X B)(F-F_X)$.
  \item[(ii)] If $K(\widehat{X},X) := K_{F_{\widehat{X}}}(X)$ and
  $H_{\widehat{X}} := H+F_{\widehat{X}}^H R F_{\widehat{X}}$,
  then \eqref{Req-a} can be rewritten as
  \begin{align}
  X-\mathcal{R}(X)&=\mathcal{S}_{T_{\widehat{X}}}(X)-H_{\widehat{X}}+K(\widehat{X},X). \label{Req-b}
  \end{align}
Furthermore, we also have
\begin{align}\label{K}
X-\mathcal{R}(X)=\mathcal{S}_{T_{\widehat{X}}}(X)-H+\left[ K(\widehat{X},X)-K(\widehat{X},0) \right].
\end{align}

  \item[(iii)] If $U_{\widehat{X}} := \widehat{X} T_{\widehat{X}}$, then $B^H U_{\widehat{X}}=R F_{\widehat{X}}$,
  $H_{\widehat{X}} = H+U_{\widehat{X}}^H G U_{\widehat{X}}$ and thus \eqref{Req-b} can be rewritten as
\begin{align}
 X-\mathcal{R}(X)&=\mathcal{S}_{T_{\widehat{X}}}(X)- H_{\widehat{X}}+(U_{\widehat{X}}-U_X)^H(G+GXG)(U_{\widehat{X}}-U_X). \label{Req-c}
  \end{align}
\end{enumerate}
\end{subequations}
\end{lemma}
\begin{proof}
\begin{enumerate}
  \item[(i)]
  Observe that
 \[
 X-\mathcal{R}(X)=\Gamma_1(X)+\Gamma_2(X),
 \]
where $\Gamma_1(X):=\mathcal{S}_A(X)-H$ and $\Gamma_2(X):=A^H X B(R+B^H X B)^{-1}B^H X A$. A direct computation yields
\begin{align*}
&\Gamma_2(X)=F_X^H(R+B^H X B)F_X=K_F(X)+[F^H(R+B^H X B)F_X\\
&+F_X^H(R+B^H X B)F]-F^H(R+B^H X B)F\\
&=K_F(X)-F^H R F+[F^HB^H X A+A^H X BF-F^HB^H X BF]\\
&=K_F(X)-F^H R F+\mathcal{S}_{A_F}(X)-\mathcal{S}_{A}(X).
\end{align*}
We conclude that
\begin{align*}
 X-\mathcal{R}(X)&=(\mathcal{S}_{A}(X)-H)+(K_F(X)-F^H R F+\mathcal{S}_{A_F}(X)-\mathcal{S}_{A}(X))\\
 &=\mathcal{S}_{A_{F}}(X)-H_F+K_F(X).
\end{align*}
\item[(ii)] The first result is clearly true from \eqref{Req-a}. For the proof of the remaining part, an easy computation shows that
    \begin{align*}
     RF_{\widehat{X}}&=B^H\widehat{X}(I+BR^{-1}B^H\widehat{X})^{-1}A= B^H\widehat{X} T_{\widehat{X}}, \\
      F_{\widehat{X}}^H R F_{\widehat{X}}&=(R F_{\widehat{X}})^H R^{-1} (R F_{\widehat{X}})=(\widehat{X}T_{\widehat{X}})^H G (\widehat{X}T_{\widehat{X}}).
    \end{align*}
    We immediately obtain this formula due to $K(\widehat{X},0)=F_{\widehat{X}}^H R F_{\widehat{X}}$.
   \item[(iii)] A trivial verification shows that
\begin{align*}
  B^HU_{\widehat{X}} &=B^H \widehat{X} T_{\widehat{X}} = (R+B^H \widehat{X} B)^{-1}B^H \widehat{X}A =R F_{\widehat{X}}, \\
  H_{\widehat{X}} &= H + F_{\widehat{X}}^H R F_{\widehat{X}} = H + U_{\widehat{X}}^H B R B^H U_{\widehat{X}} = H + U_{\widehat{X}}^HGU_{\widehat{X}}.
\end{align*}
The result \eqref{Req-c} follows from \eqref{Req-b} immediately.
\end{enumerate}
\end{proof}

In addition, under different setting, Eq.~\eqref{K} can be reformulated as follows.
\begin{subequations}\label{K123}
\begin{enumerate}
  \item If $X=\widehat{X}\in\mathbb{H}_n$, then
  \begin{equation}\label{K2}
\mathcal{S}_{T_{\widehat{X}}}(\widehat{X})=\widehat{X}-\mathcal{R}(\widehat{X})+H+K(\widehat{X},0).
\end{equation}
  \item If $X\in \mathcal{R}_=$ and $\widehat{X}\in\mathbb{H}_n$, then
\begin{equation}\label{K1}
\mathcal{S}_{T_{\widehat{X}}}(X)=H+\left[ K(\widehat{X},0)-K(\widehat{X},X)\right].
\end{equation}

  \item If $X=\widehat{X}\in \mathcal{R}_=$, then
  \begin{equation}\label{K3}
\mathcal{S}_{T_{X}}(X)=H+K(X,0).
\end{equation}
\end{enumerate}
\end{subequations}

For any $X,\widehat{X}\in\mathbb{H}_n$, consider a subset of $\mathbb{H}_n$ defined by
\begin{equation}\label{S1}
  \mathcal{S}_= := \{Y\in\mathbb{H}_n\,|\, \mathcal{S}_{T_{\widehat{X}}}(Y)=K(\widehat{X},X)\}.
\end{equation}
Clearly, from~\eqref{K2}, the solution set $\mathcal{R}_=$ has the equivalent expression given by
\begin{equation}\label{S2}
  \mathcal{R}_= = \{Y\in\mathbb{H}_n\,|\, \mathcal{S}_{T_Y}(Y)=H+K(Y,0)\}.
\end{equation}
The relationship between these two sets $\mathcal{R}_=$ and $\mathcal{S}_=$ is characterized in the following lemma,
which will be used in the proof of \cref{thm3p3} later on.
\begin{lemma} \label{lem2p7}
For $\mathcal{S}_=$ and $\mathcal{R}_=$ defined by \eqref{S1}--\eqref{S2}, the following statements hold:
\begin{enumerate}
  \item[(i)] If $X,\widehat{X}\in \mathcal{R}_=$, then $\widehat{X}-X\in \mathcal{S}_=$.
  \item[(ii)] If $X\in \mathcal{R}_=$ and $\widehat{X}-X\in \mathcal{S}_=$, then $\widehat{X}\in \mathcal{R}_=$.
  \item[(iii)] If $\widehat{X}\in \mathcal{R}_=$ and $\widehat{X}-X\in \mathcal{S}_=$, then ${X}\in \mathcal{R}_=$.
\end{enumerate}
\end{lemma}
\begin{proof}
With the aid of \eqref{K123}, the results are proven as follows.
\begin{enumerate}
  \item[(i)]
  From \eqref{K1} and \eqref{K3} we obtain
  \begin{align*}
          \mathcal{S}_{T_{\widehat{X}}}(\widehat{X}-X)&=\mathcal{S}_{T_{\widehat{X}}}(\widehat{X})-\mathcal{S}_{T_{\widehat{X}}}(X)\\
          &=H+K(\widehat{X},0)-\left[ H+K(\widehat{X},0)-K(\widehat{X},X) \right] = K(\widehat{X},X).
        \end{align*}
        That is, $\widehat{X}-X\in \mathcal{S}_=$.
  \item[(ii)]
  From \eqref{K1} we obtain
  \begin{align*}
          \mathcal{S}_{T_{\widehat{X}}}(\widehat{X})&=\mathcal{S}_{T_{\widehat{X}}}(X)+K(\widehat{X},X)\\
          &=H+K(\widehat{X},0)-K(\widehat{X},X)+K(\widehat{X},X)=H+K(\widehat{X},0)).
        \end{align*}
        That is, $\widehat{X}\in \mathcal{R}_=$.
  \item[(iii)]
  From \eqref{K3} and \eqref{K} we obtain
  \begin{align*}
          \mathcal{S}_{T_{\widehat{X}}}({X})&=\mathcal{S}_{T_{\widehat{X}}}(\widehat{X})-K(\widehat{X},X)\\
          &=H+K(\widehat{X},0)-K(\widehat{X},X)=\mathcal{S}_{T_{\widehat{X}}}({X})-\left[ X-\mathcal{R}(X) \right].
        \end{align*}
        That is, $X=\mathcal{R}(X)$ or ${X}\in \mathcal{R}_=$.
\end{enumerate}
\end{proof}

\subsection{Some useful facts and properties}
For any $M\in \mathbb{C}^{n\times n}$, the generalized eigenspace of $M$ corresponding an eigenvalue $\lambda$
is defined by $E_\lambda (M) = \mathrm{Ker}(M-\lambda I)^n$.
The following lemma characterizes the inheritance of (almost) stability property under the setting of a Stein inequality.
\begin{lemma}\label{lem2p2}
Let $B\in \mathbb{C}^{n\times n}$ and $Q\geq 0$. If $X_0$ is a positive semidefinite solution of the Stein inequality $\mathcal{S}_B(X)\geq Q$,
and $\mathrm{Ker}(Q)\subseteq \mathrm{Ker}(B-A)$ for some $A\in\mathbb{C}^{n\times n}$,
then $\rho(B)\leq \max\{1,\rho(A)\}$. Furthermore, we have
\begin{enumerate}
  \item[(i)] $\rho(B)\leq 1$ if $\rho(A)\leq 1$.
  \item[(ii)] $\rho(B)<1$ if $\rho(A)<1$ or $\mathrm{Ker}(Q)\cap E_\lambda(B)=\{0\}$ for some $\lambda\in\sigma(B)$.
\end{enumerate}
\end{lemma}
\begin{proof}
Let $\lambda\in\sigma(B)$ and a nonzero vector $x\in E_\lambda(B)$. Then
$x^H \mathcal{S}_B(X_0) x=(1-|\lambda|^2)(x^H X_0 x) \geq x^H Q x \geq 0$. If $x\not\in\mathrm{Ker}(Q)$, then $x\not\in\mathrm{Ker}(X_0)$ and thus $|\lambda|< 1$. Otherwise, $x\in \mathrm{Ker}(Q)$ implies that $Ax=Bx+(A-B)x=\lambda x$ and thus $|\lambda|\leq \rho(A)$. We complete the proof.
\end{proof}

According to the following theorem, the positive semidefiniteness of the operator $K(\cdot, \cdot)$ defined by \eqref{Req-b} might depend
on the positive definiteness of the matrix $R + B^H X B$ for any $X\in \mathcal{R}_\geq$.
\begin{theorem}\label{thm2p3}
Let $X\in \mathcal{R}_{\geq}$. Then $R+B^H X B>0$ and $K(\widehat{X},X)\geq 0$ for any $\widehat{X}\in \mathbb{H}_n$.

\end{theorem}
\begin{proof}
We first notice that $X\in \mathcal{R}_{\geq}$ is equivalent to $X\geq H+T_X^H(X+XGX)T_X$. Let $R=\widehat{R}\widehat{R}^H$ for some $\widehat{R}>0$, e.g., the Cholesky decomposition of $R$ , and $\widehat{B}:=\widehat{R}^{-H}B$. Then $G=\widehat{B}\widehat{B}^H$ and
a direct computation yields
\begin{align*}
  \begin{bmatrix} R & -B^H X T_X\\-T_X^H X B^H & X-T_X^H X T_X\end{bmatrix} &\geq \begin{bmatrix} R & -B^H X T_X\\-T_X^H X B^H & T_X^H XGX T_X\end{bmatrix} \\
  &=\begin{bmatrix} \widehat{R} \\ -T_X^H X\widehat{B}\end{bmatrix} \begin{bmatrix} \widehat{R}^H & -\widehat{B}^H X T_X \end{bmatrix} \geq 0.
\end{align*}
Therefore, we see that
\begin{align} \label{blk-ineq}
  \begin{bmatrix} R+B^H X B & 0\\0 & X\end{bmatrix} &\geq \begin{bmatrix} B^H XB  & B^H X T_X\\T_X^H X B & T_X^H X T_X\end{bmatrix}
  = U^H X U,
\end{align}
where $U := \begin{bmatrix} B & T_X \end{bmatrix}\in \mathbb{C}^{n\times (m+n)}$.
Next, we shall deduce that $R+B^H XB \geq 0$. If not, there is a nonzero $x_1\in \mathbb{C}^m$ such that $x_1^H (R+B^H XB)x_1 < 0$.
Choose a scalar $\lambda\not \in \sigma(T_X)$ with $|\lambda| = 1$, and let $x_2 := (\lambda I - T_X)^{-1}Bx_1$ and
$x:= \begin{bmatrix} x_1 \\ x_2 \end{bmatrix}$, respectively. It is easily seen that $Ux = \lambda x_2$
and hence, from \eqref{blk-ineq}, we immediately obtain
\[  x_1^H (R+B^HXB)x_1 + x_2^H Xx_2 \geq |\lambda|^2 (x_2^H Xx_2) = x_2^H X x_2. \]
Then $x_1^H (R+B^HXB)x_1\geq 0$, contradicting to the assumption of $x_1$.
Thus, we conclude that $R+B^HXB\geq 0$ is nonsingular or $R+B^H X B>0$ for any $X\in \mathcal{R}_\geq$. The remaining part immediately follows from the definition of the operator $K(.,.)$ in Lemma~\ref{lem2p1}.
\end{proof}
\begin{remark}\label{rem2p4}
From \cref{thm2p3}, we see that $I+GX=I+BR^{-1}B^HX$ is nonsingular if $X\in \mathcal{R}_\geq$ (since $I+R^{-1}B^HXB$ is nonsingular).
\end{remark}

Based on the following lemma, the uniqueness of the almost stabilizing solution to the DARE \eqref{dare} will be established
in \cref{thm3p3} below.
\begin{lemma}\label{lem2p5}
Let $\rho(A)\leq 1$ and $Q_1\geq 0$. If $X_0\in\mathbb{H}_n$ is a solution of the matrix inequality
$\mathcal{S}_A(X)\geq (X A)^HQ_1(X A)$, and the generalized eigenspaces of $A^H$ satisfy
\begin{align}\label{ass}
 E_{\bar{\lambda}}(A^H)\cap \mathrm{Ker}(Q_1)=\{0\}
\end{align}
for all $\lambda\in\sigma(A)$ with $|\lambda|=1$, then $X_0 \geq 0$.
\end{lemma}
\begin{proof}
For the sake of convenience, we factorize $Q_1$ as the Cholesky decomposition  $Q_1=Q_2 Q_2^H$. Note that $\mathrm{Ker}(Q_1)=\mathrm{Ker}(Q_2^H)$.
Let ${\lambda}\in\sigma(A)$ with $|\lambda|=1$ and $Av_1 = \lambda v_1$ for some nonzero $v_1\in \mathbb{C}^n$. Then we see that
$$0=v_1^H \mathcal{S}_A(X_0)v_1=(Q_2^H X_0v_1)^H(Q_2^H X_0v_1)=(v_1^H X_0 Q_2)(v_1^H X_0 Q_2)^H, $$
and this implies that
\[
  0=v_1^H A^H X_0Q_2 =\bar{\lambda}v_1^HX_0 Q_2, \quad
  0=v_1^H \mathcal{S}_A(X_0) =\bar{\lambda}v_1^H X_0( \lambda I-A).
\]
Thus, it follows from the assumption \eqref{ass} that $X_0v_1=0$.
Let $(A-\lambda I)v_2=v_1$. Observe that the quadratic form
$v_2^H \mathcal{S}_A(X_0)v_2$ again satisfies
\[
v_2^H(X_0-A^HX_0A)v_2=-(\lambda v_1^H X_0 v_2+\bar{\lambda}v_2^H X_0 v_1+v_1^H X_0 v_1)=0,
\]
and we then have
\[
  0=v_2^H AX_0 Q_2 =\bar{\lambda}v_2^H X_0Q_2,\quad
  0=v_2^H \mathcal{S}_A(X_0)=\bar{\lambda}v_2^H X_0(\lambda I-A).
\]
Therefore, $X_0v_2=0$ by using the assumption \eqref{ass}. Inductively, we can show that $X_0 v_j=0$ for a Jordan chain $\{v_j\}$ corresponding to the unimodular eigenvalue $\lambda$ of $A$.

{Let $J_{A}=P^{-1}A P$ be the Jordan canonical form of $A$. Suppose that $J_{A}= J_1 \oplus J_s$, where $J_1\in\mathbb{C}^{m_1\times m_1}$ with $\sigma(J_1)\subseteq \partial\mathbb{D}$ and $J_s\in\mathbb{C}^{m_2\times m_2}$ with $\sigma (J_s) \subseteq \mathbb{D}$}.
If we let $Y_0:= \mathcal{S}_{A}({X}_0)$, then $Y_0\geq (X_0 A)^HQ_1(X_0A) \geq 0$ and one deduces that
\[
\mathcal{S}_{A}(\widehat{X}_0)=\widehat{X}_0-J_A^H\widehat{X}_0 J_A=\widehat{Y}_0,
\]
where $\widehat{X}_0 := P^H X_0 P$ and $\widehat{Y}_0 := P^H Y_0 P\geq 0$.
From the above discussion we have $\widehat{X}_0= 0_{m_1}\oplus \widehat{X}_{0,s}$, $\widehat{Y}_0= 0_{m_1}\oplus \widehat{Y}_{0,s}$ and $\widehat{X}_{0,s}\in\mathbb{H}_{m_2}$ satisfies the Stein equation $\mathcal{S}_{J_s}(\widehat{X}_{0,s})=\widehat{Y}_{0,s}\geq 0$.
Therefore, $X_0=P^{-H}(0_{m_1}\oplus \widehat{X}_{0,s})P^{-1}\geq 0$,
since the matrix $\widehat{X}_{0,s}=\mathcal{S}_{J_s}^{-1}(\widehat{Y}_{0,s})$ is positive semidefinite if $\rho(J_s)<1$.
%
\end{proof}

In \cref{thm3p4}, the dual DARE of second kind will possess a positive definite stabilizing solution
if the sufficient conditions of the following lemma are fulfilled, in which
$\mathbb{D}^c = \{z\in\mathbb{C}\, |\, |z| \geq 1 \}$ is the complement of the open unit disk $\mathbb{D}$.
\begin{lemma}\label{lem2p6}
Assume that $X_0$ is a positive semidefinite solution of the Stein inequality
$ \mathcal{S}_A(X)\geq BB^H$ with $B\in\mathbb{C}^{n\times m}$. Then
\begin{enumerate}
  \item[(i)] $\rho(A)<1$ if $\mathrm{Ker}(A-\lambda I)\cap \mathrm{Ker}(B^H)=\{0\}$ for $\lambda\in\sigma(A)\cap\mathbb{D}^c$ or, equivalently,
   $\mathrm{rank}\begin{bmatrix} B^H \\ A-\lambda I \end{bmatrix}=n$ for $\lambda\in\sigma(A)\cap\mathbb{D}^c$.
  \item[(ii)] $X_0>0$ if $\mathrm{Ker}(A-\lambda I)\cap \mathrm{Ker}(B^H)=\{0\}$ for $\lambda\in\sigma(A)$ or, equivalently,
  $\mathrm{rank}\begin{bmatrix} B^H \\ A-\lambda I \end{bmatrix}=n$ for $\lambda\in\sigma(A)$.
\end{enumerate}
\end{lemma}
\begin{proof}
Under slightly different conditions, the results are derived as follows.
\begin{enumerate}
  \item[(i)] Let $u\in\mathrm{Ker}(A-\lambda I)$ with $u\neq 0$ for $\lambda\in\sigma(A)\cap\mathbb{D}^c$. Then
\begin{align*}
  0&\geq (1-|\lambda|^2)u^H X_0 u\geq (B^H u)^H (B^H u) \geq 0.
\end{align*}
From the assumption we conclude that $u\in\mathrm{Ker}(A-\lambda I)\cap \mathrm{Ker}(B^H)=\{0\}$, which leads to a contradiction with $u\neq 0$.
Thus, $|\lambda|<1$ for all $\lambda\in\sigma(A)$ or, equivalently, $\rho(A)<1$.
\item[(ii)] Let $\lambda\in\sigma(A)$. For any $v_1\in\mathrm{Ker}(A-\lambda I)\cap \mathrm{Ker}(X_0)$, we see that
\begin{equation*}
   0 = (1-|\lambda|^2)v_1^H X_0 v_1\geq (B^H v_1)^H (B^H v_1) \geq 0.
\end{equation*}
Then $v_1\in\mathrm{Ker}(A-\lambda I)\cap \mathrm{Ker}(B^H)=\{0\}$ and thus $v_1=0$.
This implies that $\mathrm{Ker}(A-\lambda I)\cap \mathrm{Ker}(X_0)=\{0\}$.
Suppose that $\mathrm{Ker}(A-\lambda I)^i \cap \mathrm{Ker}(X_0)=\{0\}$ for $1\leq i\leq k$ and $k\in \mathbb{N}$.
For any $v_{k+1}\in \mathrm{Ker}(A-\lambda I)^{k+1} \cap \mathrm{Ker}(X_0)$, let $v_k:= (A-\lambda I)v_{k+1}$ or $Av_{k+1}=v_k+\lambda v_{k+1}$.
Thus, we further obtain
 \begin{align*}
  0&\geq -v_k^H X_0 v_k=v_{k+1}^H X_0 v_{k+1}-(v_k^H+\bar{\lambda}v_{k+1}^H)X_0(v_k+\lambda v_{k+1}) \\
  &=v_{k+1}^H \mathcal{S}_A(X_0)v_{k+1} \geq (B^H v_{k+1})^H (B^H v_{k+1}) \geq 0.
\end{align*}
This leads to $X_0v_k=0$ and hence $v_k\in \mathrm{Ker}(A-\lambda I)^k \cap \mathrm{Ker}(X_0)=\{0\}$ or $v_k=0$.
Then $v_{k+1}\in\mathrm{Ker}(A-\lambda I)\cap \mathrm{Ker}(B^H)=\{0\}$ or $v_{k+1} = 0$.
That is, $\mathrm{Ker}(A-\lambda I)^{k+1}\cap \mathrm{Ker}(X_0)=\{0\}$. It follows from mathematical induction that
$\mathrm{Ker}(A-\lambda I)^\ell \cap \mathrm{Ker}(X_0)=\{0\}$ for each $\ell \in\mathbb{N}$.
Therefore, $\mathrm{Ker}(X_0)=\mathrm{Ker}(X_0)\cap \mathbb{C}^n
=\mathrm{Ker}(X_0)\cap \Big( \bigoplus\limits_{\lambda\in\sigma(A)}E_{\lambda}(A) \Big)=\{0\}$
or, equivalently, $X_0\geq 0$ is nonsingular.
\end{enumerate}
\end{proof}
\subsection{The formulation of the dual DARE} \label{sec3p2}
Let $A$ be nonsingular through this
subsection. The positive semidefinite solutions of dual DARE play a central role in the set of negative semidefinite solutions of the original DARE. For the DARE \eqref{dare-b} of the compact form, we shall provide the construction of two kinds of dual DARE \eqref{dual} and \eqref{DY}, in terms of the coefficient matrices $A$, $G$ and $H$, in this subsection.
\subsubsection{The first kind of dual DARE}
For the sake of simplicity the matrix products $A^{-H}X A^{-1}$ is denoted by $X^{(A)}$ for any $X\in\mathbb{C}^{n\times n}$. Note that $I+GH^{(A)}$ is nonsingular since $G,\,H^{(A)}\geq 0$. Then Eq.~\eqref{dare-b} is equivalent to
     \begin{align*}
   (I+GX)^{-1}&=I-G(X-H)^{(A)}=I+GH^{(A)}-GX^{(A)}\\
   &=(I+GH^{(A)})A(I-A^{-1}(I+GH^{(A)})^{-1}GA^{-H}X)A^{-1}\\
   &=\widehat{A}^{-1}(I-\widehat{G}X)A^{-1}.
     \end{align*}
     We immediately obtain $I-\widehat{G}X$ is nonsingular and
     \begin{align}\label{pd}
       [(I+GX)^{-1}A] \times  [(I+\widehat{G}Y)^{-1}\widehat{A}]=I
     \end{align}
       , where $Y=-X$, $\widehat{A}=A^{-1}(I+H^{(A)})^{-1}$ and $\widehat{G}=((I+GH^{(A)})^{-1}G)^{(A^H)}=\widehat{A}GA^{-H}$.
   Eq.~\eqref{pd} provides the formulation of the first kind of dual DARE
   \begin{align}\label{dual}
   Y=\mathcal{D}_1(Y):=\widehat{H}+\widehat{A}^H Y(I+\widehat{G}Y)^{-1}\widehat{A},
   \end{align}
   where
   \begin{align*}
     \widehat{H} &= -X+\widehat{A}^H X(I-\widehat{G}X)^{-1}\widehat{A}=-X+\widehat{A}^H XA^{-1}(I+GX)\\
     &=-X+\widehat{A}^H H A^{-1}(I+GX)+(I+H^{(A)}G)^{-1} X\\
     &=\widehat{A}^H H A^{-1}+\Big[-I+\widehat{A}^H H A^{-1}G+(I+H^{(A)}G)^{-1}\Big]X \\
     &=\widehat{A}^H H A^{-1}+\Big[-I+(I+H^{(A)}G)^{-1}H^{(A)}G+(I+H^{(A)}G)^{-1}\Big]X \\
     &=\widehat{A}^H H A^{-1}.
    \end{align*}
Notice that if we write $G = BR^{-1}B^H$, then it is easily seen that
\begin{subequations}\label{dual3}
\begin{align}
  \widehat{A} &= A^{-1} - \widehat{B} \widehat{R}^{-1} B^H H^{(A)}, \\
   \widehat{G} &= \widehat{B}\widehat{R}^{-1} \widehat{B}^H\geq 0, \\
  \widehat{H} &= H^{(A)} - B^H H^{(A)}\widehat{R}^{-1} H^{(A)} B,
\end{align}
\end{subequations}
where $\widehat{B} = A^{-1}B$ and
   $\widehat{R}= R + B^H  H^{(A)} B>0$. \eqref{dual3} are the same as the coefficient matrices defined in \cite[Theorem 3.1]{i96}. More precisely, we have
\[
\widehat{A}=\widetilde{A}-\widetilde{B}\widetilde{R}^{-1}\widetilde{C},\widehat{H}=\widetilde{H}-\widetilde{C}^H\widetilde{R}^{-1}\widetilde{C}.
\]
Finally, it is straightforward to verify
 \begin{align}\label{dual1}
 (\mathcal{R}(X)-X)^{(A)}=(\mathcal{D}_1(Y)-Y)(I+GH^{(A)}),
 \end{align}
 i.e., $Y=-X$ is a solution of dual DARE~\eqref{dual} if and only if $X$ is a solution of DARE~\eqref{dare-b}.  Notice that from \eqref{pd} we have
\begin{align*}
\sigma((I+\widehat{G}Y)^{-1}\widehat{A})=\sigma(T_X^{-1}).
\end{align*}
\subsubsection{The second kind of dual DARE} \label{sec3p3}
Let $X$ be a Hermitian solution of the DARE \eqref{dare-b},
then $X-H$ is nonsingular and
\begin{equation} \label{eyeX}
\left[ X(I+GX)^{-1}A\right] \times \left[ (X-H)^{-1} A^H \right] = I.
\end{equation}
Furthermore, if $X$ is nonsingular, it follows from \eqref{eyeX} that $Y:= -X^{-1}$ satisfies
\begin{equation} \label{invCLM}
  \left[ X(I+GX)^{-1}A \right]\times \left[Y (I+HY)^{-1}A^H\right]=-I
\end{equation}
or, equivalently, $\left[Y (I+HY)^{-1}A^H\right]\times\left[ X(I+GX)^{-1}A \right]=-I$.
It thus implies that $Y$ is a Hermitian solution to the second kind of dual DARE
\begin{equation} \label{DY}
  Y =\mathcal{D}_2(Y):= A Y (I + H Y)^{-1} A^H + G.
\end{equation}
In addition, it is easily seen that the Hermitian matrices $X$ and $Y$ also fulfill
\begin{subequations}\label{dual2}
\begin{align}
  Y - \mathcal{D}_2(Y) &= A\left[ (X-H)^{-1} - (\mathcal{R}(X)-H)^{-1} \right] A^H \\
  X - \mathcal{R}(X) &= A^H \left[ (Y-G)^{-1} - (\mathcal{D}_2
  (Y)-G)^{-1} \right] A,
\end{align}
\end{subequations}

i.e., $Y=-X^{-1}$ is a nonsingular solution of dual DARE~\eqref{DY} if and only if $X=-Y^{-1}$ is a nonsingular solution of DARE~\eqref{dare-b}. Notice that from \eqref{invCLM} we have
\begin{align}\label{dualcloseloop}
\sigma(((I+HY)^{-1}A^H))=\sigma(X T_X^{-1} X^{-1})=\sigma(T_X^{-1}).
\end{align}
Analogously, let
$\mbox{dom}(\mathcal{D}_1):=\{X\in\mathbb{H}_n\, |\, \det(I+\widehat{G}X)\neq 0\}$ , $\mbox{dom}(\mathcal{D}_2):=\{X\in\mathbb{H}_n\, |\, \det(I+HX)\neq 0\}$
and $\mathcal{D}_\geq^{(i)} := \{X\in \mbox{dom}(\mathcal{D}_i)\, |\, X \geq \mathcal{D}_i(X) \}$,
with $i=1,2$  for the sake of explanation.

\section{Extremal solutions of the DARE} \label{sec3}
In this section, the existence of extremal solutions to the DARE \eqref{dare}
will be established iteratively through the FPI given by
\begin{equation} \label{fpi-psd}
X_{k+1} = \mathcal{R}(X_k), \quad k\geq 0,
\end{equation}
with a suitable initial matrix $X_0\in \mathbb{N}_n$.

According to the formulations~\eqref{dual1} and \eqref{dual2}, when $A$ is nonsingular, it is shown that $X\in \mathbb{H}_n$ is a solution of the DARE \eqref{dare}
if and only if $-X\in \mathbb{H}_n$ is a solution of its dual DARE of first kind \eqref{dual1},
and $X\in \mathbb{H}_n$ is a nonsingular solution of the DARE \eqref{dare}
if and only if $-X^{-1}\in \mathbb{H}_n$ is a solution of its dual DARE \eqref{DY} of second kind.
 With this reason, we thus consider two FPIs defined by
\begin{subequations}\label{fpi-nsd}
 \begin{align}
  Y_{k+1} &= \mathcal{D}_1 (Y_k),\quad k\geq 0,\label{fpi-nsd1}\\
  Z_{k+1} &= \mathcal{D}_2 (Z_k),\quad k\geq 0,\label{fpi-nsd2}
\end{align}
\end{subequations}
with $Y_0,Z_0\in \mathbb{N}_n$ being the initial guesses.

For the sake of explanation, $X_{+,M}$, $X_{+,m}$, $X_{-,M}$ and $X_{-,m}$ denote the maximal positive semidefinite solution, minimal positive semidefinite solution, maximal negative semidefinite solution and minimal negative semidefinite solution of DARE, respectively, if they exist. For the sake of clarity, posive semidefinite and negative semidefinite extremal solutions of the DARE~\eqref{dare} will be discussed separately in the following subsections.

\subsection{Positive semidefinite extremal solutions} \label{sec3p1}
The following theorem, quoted from \cite[Theorem 3.2]{chiang2021}, guarantees the existence of the minimal solution $X_{+,m} \in \mathbb{N}_n$ to the DARE \eqref{dare-b},
i.e., $X_{+,m} \leq S$ for all $S\in \mathcal{R}_\geq\cap \mathbb{N}_n$, under reasonable assumptions.
\begin{theorem}\cite{chiang2021}\label{thm3p1}
  If $\mathcal{R}_{\geq}\cap\mathbb{N}_n \neq \emptyset$, then the FPI \eqref{fpi-psd}
  generates a nondecreasing sequence of positive semidefinite matrices $\{X_k\}_{k=0}^\infty$ with $0 \leq X_0 \leq H$,
  which converges at least R-linearly to the minimal element $X_{+,m}$ of the solution set  $\mathcal{R}_\geq\cap \mathbb{N}_n$
  with the rate of convergence
  \[  \limsup_{k\rightarrow \infty} \sqrt[k]{\|X_k - X_{+,m}\|}\leq \rho_{\mathbb{D}} (T_{X_{+,m}})^2 < 1. \]
\end{theorem}

As a special case of Theorem 1.1 in \cite{g98}, if $(A,B)$ is stabilizable and $H\geq 0$, then
the DARE \eqref{dare-a} has a unique almost stabilizing $X_{\rm s} \in \mathbb{H}_n$,
which is the maximal element of the solution set $\mathcal{R}_=$ satisfying $R+B^H X_{\rm s} B > 0$.
In the following theorem the existence and uniqueness of the almost stabilizing solution to the DARE \eqref{dare}
can be established under the framework of the fixed-point iteration.
\begin{theorem} \label{thm3p2}
If there exists $X_\star \in \mathbb{H}_n$ satisfying $\rho (T_{X_\star}) < 1$, then the following statements hold:
\begin{enumerate}
  \item[(i)] $\mathcal{S}_\geq = \{ X\in \mathbb{H}_n\, |\, \mathcal{S}_{T_{X_\star}} (X)\geq H_{X_\star} \}$ is a nonempty subset of $\mathcal{R}_\geq\cap\mathbb{N}_n$, where $H_{X_\star}=H+F_{X_\star}^H R F_{X_\star}$.
  \item[(ii)] The FPI \eqref{fpi-psd} generates a nonincreasing sequence of positive semidefinite matrices $\{X_k\}_{k=0}^\infty$ with $X_0 = \mathcal{S}_{T_{X_\star}}^{-1} (H_{X_\star})$,
  which converges at least R-linearly to the maximal element $X_{+,M}$ of $\mathcal{R}_\geq\cap \mathbb{N}_n$ with the rate of convergence
  \[  \limsup_{k\rightarrow \infty} \sqrt[k]{\|X_k - X_{+,M}\|}\leq \rho (T_{+,M})^2, \]
 provided that $\rho (T_{+,M})<1$.
\item[(iii)] For each $k\geq 0$, $\rho(T_{X_k}) < 1$ and thus $\rho (T_{+,M}) \leq 1$.
\end{enumerate}
\end{theorem}
\begin{proof}
Notice that the operator $\mathcal{S}_{T_{X_\star}}$ is invertible because $\rho (T_{X_\star}) < 1$.
  \begin{enumerate}
    \item[(i)] Since $H_{X_\star} \geq 0$, it follows from \cref{lem2p1} that $X_0 = \mathcal{S}_{T_{X_\star}}^{-1} (H_{X_\star})\in \mathcal{S}_\geq \cap \mathbb{N}_n$,
    and we further deduce that
    \[  X-\mathcal{R}(X) = \mathcal{S}_{T_{X_\star}}(X) - H_{X_\star} + K(X_\star, X)\geq 0 \]
   and $X\geq 0$ for all $X\in \mathcal{S}_\geq$, since $K(X_\star,X)\geq 0$ in \eqref{Req-b} for $X\geq 0$.
  \item[(ii)] Because $X_0 = \mathcal{S}_{T_{X_\star}}^{-1} (H_{X_\star}) \geq 0$ is an element of the set $\mathcal{R}_\geq$ and $H\geq 0$, we see tat $X_0\geq \mathcal{R}(X_0) = X_1\geq 0$ and hence
  it is easily seen that $X_k\geq X_{k+1}\geq 0$ for $k\geq 1$ by induction, since the operator $\mathcal{R}(\cdot)$ is order preserving.
  Thus, $\{X_k\}_{k=0}^\infty$ generated by the FPI \eqref{fpi-psd} is
  a nonincreasing sequence of positive semidefinite matrices,
  which converges to a limit $X_{+,M} \in \mathbb{N}_n$ eventually.
  On the other hand, for any $X_+\in \mathcal{R}_=$, it follows from \cref{thm2p3} and \eqref{Req-b} that $K(X_\star, X_+)\geq 0$ and
 $H_{X_\star} = \mathcal{S}_{T_{X_\star}}(X_+) + K(X_\star, X_+) \geq  \mathcal{S}_{T_{X_\star}}(X_+)$.
  Since the inverse of $\mathcal{S}_{T_{X_\star}}$ is order preserving, we see that $X_0 = \mathcal{S}_{T_{X_\star}}^{-1} (H_{X_\star}) \geq X_+$.
  Suppose that $X_i\geq X_+$ for $i=0,1,2,\ldots, k$. Then $X_{k+1} - X_+ = \mathcal{R}(X_k) - \mathcal{R}(X_+)\geq 0$ and
  we thus deduce that $X_k \geq X_+$ for all $k$. This yields $X_\infty:= \lim\limits_{k\rightarrow \infty} X_k \geq X_+$ for all $X_+\in \mathcal{R}_=$,
  i.e., the limit of $\{X_k\}_{k=0}^\infty$ must be the maximal element of the solution set $\mathcal{R}_=$ or $X_{+,M}=X_{\infty}$.
  Moreover, the proof of the rate of convergence for $\{X_k\}_{k=0}^\infty$ follows from the Appendix of \cite{l.c18}.
  \item[(iii)] For each $k\geq 0$, from \eqref{Req-b} we see that
\begin{subequations}\label{AP}
\begin{align}
X_k-X_{k+1} &= \mathcal{S}_{T_{X_k}} (X_k) - H_{X_k} \label{XX1} \\
&=\mathcal{S}_{T_{X_{k-1}}} (X_k) - H_{X_{k-1}} + K(X_{k-1}, X_k), \label{XX2}
\end{align}
\end{subequations}
if we let $X_{-1} := X_\star \in \mathbb{H}_n$. According to the Eq.~\eqref{XX1}, one deduces that $X_{k+1} = T_{X_k}^H X_k T_{X_k} + H_{X_k}$ for $k\geq 0$,
and it follows from \eqref{XX2} that
\[
\mathcal{S}_{T_{X_{k-1}}} (X_k) - H_{X_{k-1}} \geq X_k - T_{X_{k-1}}^H X_{k-1} T_{X_{k-1}} - H_{X_{k-1}} = 0
\]
for all $k\geq 1$. Therefore, this yields the following Stein inequality
\[
\mathcal{S}_{T_{X_k}}(X_k) \geq  H_{X_k} + K(X_{k-1}, X_k) \geq 0
\]
for each $k\geq 0$, where the equality  holds for $k=0$. Furthermore, we also have
\begin{align*}
\mathrm{Ker} \left(K(X_{k-1}, X_k)\right) &= \mathrm{Ker}\left(B^H(U_{X_{k-1}} - U_{X_{k}})\right) = \mathrm{Ker}(F_{X_{k-1}} - F_{X_k}) \\
&\subseteq \mathrm{Ker}(BF_{X_{k-1}} - BF_{X_k}) = \mathrm{Ker}(T_{X_{k-1}} - T_{X_k}).
\end{align*}
Applying \cref{lem2p2}, we conclude that $\rho(T_{X_k})<1$, if $\rho(T_{X_{k-1}})<1$, for each $k\geq 0$.
  \end{enumerate}
\end{proof}

Now, in the following theorem, we will give some equivalent conditions for the stabilizability of the pair $(A,B)$
from a differnt point of view.
\begin{theorem}\label{thm3p3}
The following statements are equivalent:
\begin{enumerate}
  \item[(i)] The pair $(A,B)$ is stabilizable.
  \item[(ii)] Tere exists a matrix $X_\star \in \mathbb{H}_n$ satisfying $\rho (T_{X_\star}) < 1$.
  \item[(iii)] The DARE \eqref{dare} has a unique almost stabilizing solution $X\in \mathbb{H}_n$.
\item[(iv)] The DARE \eqref{dare} has a maximal and almost stabilizing solution $X\in \mathbb{H}_n$.
\end{enumerate}
\end{theorem}
\begin{proof}
The equivalence of these statements is given as follows.
\begin{enumerate}
  \item[(i)$\Rightarrow$(ii):] If $(A,B)$ is stabilizable, then $A_F= A-BF$ is d-stable, i.e., $\rho (A_F) < 1$, for some $F\in \mathbb{C}^{m\times n}$.
Recall that $H_F= H + F^H R F \geq 0$, then $X_\star = \mathcal{S}_{A_F}^{-1} (H_F) \geq 0$ and we further have
\begin{align*} 
  \mathcal{S}_{T_{X_\star}} (X_\star) &= (F - F_{X_\star})^H (I + B^H X_\star B) (F - F_{X_\star}) + H + F_{X_\star}^H F_{X_\star} \\
  &\geq (F - F_{X_\star})^H (I + B^H X_\star B) (F - F_{X_\star}) \geq 0,
\end{align*}
Since $\mathrm{Ker}(F - F_{X_\star})\subseteq \mathrm{Ker}(A_F - T_{X_\star})$,
it follows from \cref{lem2p2} and $\rho (A_F) < 1$ that the closed-loop matrix associated with $X_\star\in \mathbb{H}_n$ is also d-stable, i.e.,
$\rho (T_{X_\star}) < 1$.
  \item[(i)$\Leftarrow$(ii):] Suppose that the statement (ii) holds. If we let $F = (R+B^H X_\star B)^{-1}B^H X_\star A$, then
  $\rho (A-BF) = \rho (T_{X_\star}) < 1$ and hence $(A,B)$ is stabilizable.
  \item[(i)$\Rightarrow$(iii):] Suppose that $(A,B)$ is stabilizable. Firstly, it follows from \cref{thm3p2} that the DARE \eqref{dare} has an almost stabilizing solution $X\geq 0$,
      which is the maximal element of the solution set $\mathcal{R}_=$, if either the statement (i) or (ii) holds.
      For the uniqueness of the solution $X$, we assume that the DARE \eqref{dare} has another solution
      $\widehat{X}\in \mathbb{H}_n$ with $\rho (T_{\widehat{X}}) \leq 1$. It is clear that $\widehat{X} \leq X$ because of the maximality of $X$.
      On the other hand, from \cref{lem2p5} and let
      $\Delta := \widehat{X} - X$, we have $\mathcal{S}_{T_{\widehat{X}}}(\Delta) =K(\widehat{X},X)$. From $F_{\widehat{X}}-F_{X}=(R+B^H {X} B)^{-1}B^H\Delta T_{{X}}$ we deduce that
    \[  \mathcal{S}_{T_{\widehat{X}}}(\Delta) =K(\widehat{X},X)=T_{\widehat{X}}^H\Delta Q_1\Delta T_{\widehat{X}},
      \]
      with $Q_1:= B(R+B^H X B)^{-1}B^H\geq 0$. Since $(A,B)$ is stabilizable and $T_{\widehat{X}} = A - BF_{\widehat{X}}$,
      it follows that $\mathrm{rank} [T_{\widehat{X}} - \lambda I\ B] = n$ for all $\lambda \in \mathbb{D}^c$, and hence $\Delta \geq 0$ or
      $\widehat{X} \geq X$ follows from \cref{lem2p5}. Therefore, the DARE \eqref{dare} must have
      a unique almost stabilizing solution $X=\widehat{X}$.
\item[(i)$\Leftarrow$(iii):] Suppose that the statement (iii) holds and that $(A,B)$ is not stabilizable. Then $\mathrm{rank}[A-\lambda I, B] < n$ for some $\lambda \in \mathbb{C}$ with $|\lambda | \geq 1$.
If $|\lambda | > 1$ , then $\mathrm{rank}(A-BF_X - \lambda I) < n$ or $\det (T_X - \lambda I) = 0$, and we thus have
$\lambda\in \sigma (T_X)$, which leads a contradiction with $\rho (T_X) \leq 1$.
On the other hand, if $\mathrm{rank}[A-\lambda I, B] < n$ for some $\lambda \in \mathbb{C}$ with $|\lambda | = 1$, there is a nonzero $y\in \mathbb{C}^n$
such that $y^H [A-\lambda I \ B] = 0$. Hence this implies that $\Delta := yy^H$ satisfies the following equation
\[  \mathcal{S}_{T_X}(\Delta) = T_{\Delta}^H \Delta BR^{-1}B^H\Delta T_{\Delta} = K(\Delta, 0).
\]
Therefore, it follows from \cref{lem2p7} that $\widehat{X} := X+\Delta \in \mathbb{H}_n$ is also a solution of the DARE \eqref{dare}
with $\sigma (T_{\widehat{X}}) = \sigma (T_X) \subseteq \bar{\mathbb{D}}$, which leads to a contradiction with the uniqueness of
the almost stabilizing solution $X$ to the DARE \eqref{dare}.
Consequently, the pair $(A,B)$ must be stabilizable.
 \item[(i)$\Rightarrow$(iv):] This follows from \cref{thm3p2} directly, since the statements (i) and (ii) are equivalent.
 \item[(i)$\Leftarrow$(iv):] Assume that the DARE \eqref{dare} has an almost stabilizing and maximal solution $X\in \mathbb{H}_n$.
 If $(A,B)$ is not stabilizable, applying the similar arguments described previously,
 there exist a nonzero $y\in \mathbb{C}$ and a matrix $\Delta  = yy^H \geq 0$ such that $X + \Delta$ is also an almost stabilizing solution
 of the DARE \eqref{dare-a}. Because of the maximality of $X$, we obtain $X+\Delta \leq X$ and hence $\Delta  =0$,
 which leads to a contradiction with $y \neq 0$. Thus $(A,B)$ is stabilizable.
\end{enumerate}
\end{proof}

\subsection{Negative semidefinite extremal solutions} \label{sec3p4}
In this section, when $A\in \mathbb{R}^{n\times n}$ is assumed to be nonsingular, the existence of the negative semidefinite extremal solutions
to the DARE \eqref{dare} will be addressed under the framework of two FPIs \eqref{fpi-nsd}.

\subsubsection{The first approach}
Recall that the matrices $\widehat{A}$ and $\widehat{B}$ are defined by \eqref{dual3}. For a nonzero $\lambda$, the block row matrix $\begin{bmatrix}  \widehat{A}-\lambda I & \widehat{B}\end{bmatrix}$ can be decomposed into
\[
\begin{bmatrix} \widehat{A}-\lambda I & \widehat{B}\end{bmatrix}=A^{-1}\begin{bmatrix}  A-\frac{1}{\lambda}I & B\end{bmatrix} \begin{bmatrix} -\lambda I & 0\\0 & I \end{bmatrix} \begin{bmatrix}  I & 0 \\ -\widehat{R}^{-1}B^H H^{(A)} & 0\end{bmatrix}.
\]
Thus, we see that the pair $(\widehat{A}, \widehat{B})$ is stabilizable if and only if $\mathrm{rank}[A-\lambda I\ B] = n$ for $\lambda\in \bar{\mathbb{D}}\backslash \{0\}$. Suppose that $\mathrm{rank}[A-\lambda I\ B] = n$ for all $\lambda\in \bar{\mathbb{D}}\backslash \{0\}$. Due to the stabilizability of $(\widehat{A}, \widehat{B})$ there exists a $\widehat{F}\in\mathbb{C}^{m\times n}$ such that the matrix $\widehat{A}_{\widehat{F}}:\widehat{A}-\widehat{B}\widehat{F}$ satisfies $\sigma(\widehat{A}_{\widehat{F}})\subseteq \mathbb{D}$.


The following results can be proven by applying the similar arguments described in the previous sections.
Therefore, we state two theorems without proof as follows.
\begin{theorem}\label{thm3p1-dual}
Assume that $\widehat{H}\geq 0$. Then the following statements hold:
\begin{itemize}
  \item[(i)] If $\mathcal{D}_{\geq}^{(1)}\cap \mathbb{N}_n \neq \emptyset$, then FPI \eqref{fpi-nsd1} generates a nondecreasing sequence of positive semidefinite matrices $\{{Y}_k\}_{k=0}^\infty$ with $0 \leq {Y}_0 \leq \widehat{H}$, which converges at least R-linearly to $-X_{-,M}$ with the rate of convergence
  \[  \limsup_{k\rightarrow \infty} \sqrt[k]{\|{Y}_k + X_{-,M}\|}\leq \rho_{\mathbb{D}} (T_{X_{-,M}}^{-1})^2 < 1. \]
  \item[(ii)]
If $\mathrm{rank}[A-\lambda I\ B] = n$ for all $\lambda\in \bar{\mathbb{D}}\backslash \{0\}$, then the FPI \eqref{fpi-nsd} generates a nonincreasing sequence of positive semidefinite matrices $\{{Y}_k\}_{k=0}^\infty$
with ${Y}_0 =\mathcal{S}_{\widehat{A}_{\widehat{F}}}(\widehat{H}_{\widehat{F}})\geq 0$, which converges to $-X_{-,m}$
  with the rate of convergence
  \[  \limsup_{k\rightarrow \infty} \sqrt[k]{\|{Y}_k +X_{-,m}\|}\leq \rho (T_{X_{-,m}}^{-1})^2=\mu(T_{X_{-,m}})^2\leq 1,  \]
 where $\widehat{H}_{\widehat{F}}=\widehat{H}+\widehat{F}^H R \widehat{F}$.
 \end{itemize}
\end{theorem}

\subsubsection{The second approach}
For solving the dual DARE \eqref{DY} of second kind, we will apply FPI~\eqref{fpi-nsd2} given by
\begin{equation*} 
  Z_{k+1} = \mathcal{D}_2(Z_k) = A Z_k (I + H Z_k)^{-1} A^H + G,\quad k \geq 0,
\end{equation*}
where $Z_0\geq 0$ is some initial matrix, $G = BR^{-1}B^H$ and $H\geq 0$, respectively.
The following result guarantees that the limit of the sequence $\{Z_k\}_{k=0}^\infty$, if it exists, will be a unique (almost) stabilizing solution of the dual DARE \eqref{DY} under wild assumptions.
\begin{lemma} \label{thm3p4}
Assume that $D_\geq^{(2)}\cap\mathbb{N}_n \neq \emptyset$.
Let $\{Z_k\}_{k=0}^\infty$ be a sequence generated by the FPI \eqref{fpi-nsd2}.
Then the following statements hold:
\begin{enumerate}
  \item[(i)] If $0\leq Z_0\leq G$, then $\{Z_k\}_{k=0}^\infty$ is a nondecreasing sequence of positive semidefinite matrices,
  which converges to the minimal element $Z_\infty$ of the set $D_\geq^{(2)}\cap\mathbb{N}_n$.
  \item[(ii)] Assume that $(A,B)$ is controllable, where $A\in \mathbb{C}^{n\times n}$ and $B\in \mathbb{C}^{n\times m}$.
  Then $Z_k > 0$ for $k\geq n$ if $Z_0 = 0$, and $Z_k > 0$ for $k\geq n-1$ if $Z_0 = G$.
Moreover, the dual DARE \eqref{DY} has a unique almost stabilizing solution $Z_\infty > 0$.
\item[(iii)] If $0\leq Z_0\leq G$ and $(A(I + Z_\infty H)^{-1}, B)$ is controllable, where $Z_\infty$ is the minimal element of $D_\geq^{(2)}\cap\mathbb{N}_n$,
then $Z_\infty$ is a positive definite stabilizing solution to the dual DARE \eqref{DY}.
\end{enumerate}
 \end{lemma}
\begin{proof}
\par\noindent
\begin{enumerate}
\item[(i)] This statement follows from \cref{thm3p1} directly.
\item[(ii)] Suppose that the pair $(A,B)$ is controllable. It is well known that the matrix
$B_n := [B, AB, \cdots, A^{n-1}B]$ has full rank. Now, if $Z_0 = 0$, then $Z_1 = BB^H := B_1R_1^{-1}B_1^H$,
where $B_1 = B$ and $R_1 = I_m$ is the $m\times m$ identity matrix.
Thus we further see that
\begin{align*}
  Z_2 &= \mathcal{D}_2(Z_1) = AB_1R_1^{-1}B_1^H (I + H B_1R_1^{-1}B_1^H)^{-1}A^H + BB^H \\
&= BB^H + AB_1 (R_1 + B_1^H H B_1)^{-1}B_1^H A^H \\
&= \begin{bmatrix} B_1 & AB_1 \end{bmatrix}
\begin{bmatrix} R_1 & 0\\ 0 & R_1+B_1^H H B_1 \end{bmatrix}^{-1}
\begin{bmatrix} B_1^H \\ B_1^H A^H \end{bmatrix} := B_2 R_2^{-1} B_2,
\end{align*}
with $B_2 := [B_1\ AB_1] = [B\ AB]$ and $R_2:= R_1 \oplus (R_1+B_1^H H B_1) > 0$.
Inductively, we obtain $Z_{k+1} = B_{k+1} R_{k+1}^{-1}B_{k+1}^H$,
where $B_{k+1} := [B_1\ AB_k] = [B\ AB\ A^2B\ \cdots  \ A^kB]$ and
$R_{k+1} := R_1\oplus (R_k + B_k^H H B_k)  >0$ for $k\geq 1$.
In particular, when $k=n-1$, we see that $Z_n = B_n R_n^{-1}B_n^H$ with $\mathrm{rank}(B_n) = n$.
Let $B_n^H = Q \begin{bmatrix} L^H \\ 0 \end{bmatrix}$ be the QR-decomposition with $Q\in \mathbb{C}^{mn\times mn}$ and $L\in \mathbb{C}^{n\times n}$
being unitary and nonsingular lower-triangular matrices, respectively. This implies that $Z_n = [L\ 0] (Q^H R_n^{-1}Q)\begin{bmatrix} L^H \\ 0 \end{bmatrix} > 0$.
Furthermore, continuing this process, it is still valid that $Z_k > 0$ for each $k\geq n$.
For the case that $Z_0 = G$, the result can be shown similarly.
Since $\{Z_k\}_{k=0}^\infty$ is a nondecreasing sequence, we see that $Z_k \geq Z_{k_0} > 0$ for some positive integer $k_0$ and
thus $Z_\infty = \lim\limits_{k\rightarrow \infty} Z_k \geq Z_{k0} > 0$.
Let $W_{Z_\infty} := (I+HZ_\infty)^{-1}A^H$ be the closed-loop matrix associated with $Z_\infty$ and we shall prove that $\rho (W_{Z_\infty}) \leq 1$.
If not, there is a $\lambda\in \mathbb{C}$ and a nonzero $y\in \mathbb{C}^n$ such that $|\lambda| = \rho (W_{Z_\infty}) > 1$ and $W_{Z_\infty} y = \lambda y$.
Note that the matrix $Z_\infty > 0$ also satisfies
\begin{equation} \label{stein-Ginf}
  \mathcal{S}_{W_{Z_\infty}} (Z_\infty) = W_{Z_\infty}^H (Z_\infty + Z_\infty H Z_\infty) W_{Z_\infty} + BR^{-1}B^H \geq BB^H.
  \end{equation}
Thus this implies $0 > (1 - |\lambda|^2)(y^H Z_\infty y) \geq y^H (BR^{-1}B^H)y$,
which leads to a contradiction with the positive semiidefiniteness of $G=BR^{-1}B^H$.
Consequently, $Z_\infty > 0$ is the unique almost stabilizing solution of the dual DARE \eqref{DY}.
\item[(iii)] Let $W_{Z_\infty}$ be the closed-loop matrix defined previously.
Since $(W_{Z_\infty}^H, B)$ is controllable by the assumption, it follows that
$\mathrm{rank} \begin{bmatrix} W_{Z_\infty} - \lambda I \\ B^H \end{bmatrix} = n$ for all $\lambda \in \sigma (W_{Z_\infty})$.
Thus, from \cref{lem2p6} and \eqref{stein-Ginf}, we obtain $\rho (W_{Z_\infty}) < 1$ and $Z_\infty > 0$ immediately.
\end{enumerate}
\end{proof}
Based on (i) and (ii) of \cref{thm3p4} and \eqref{dualcloseloop}, the minimal negative semidefinite solution of the DARE (1.1) can
be obtained from the limit of $\{Z_k\}_{k=0}^\infty$, which is summarized in the following theorem.
\begin{theorem} \label{thm3p7}
Assume that $D_\geq^{(2)}\cap\mathbb{N}_n\neq \emptyset$. If $(A,B)$ is controllable, then FPI \eqref{fpi-nsd2} generates a nondecreasing sequence of positive semidefinite matrices $\{{Z}_k\}_{k=0}^\infty$ with ${Z}_0=0$, which converges at least R-linearly to a positive definite matrix $Z_\infty$ with the rate of convergence
  \[  \limsup_{k\rightarrow \infty} \sqrt[k]{\|{Z}_k -Z_\infty\|}\leq \rho_{\mathbb{D}} (T_{Z_\infty}^{-1})^2 < 1. \]
Furthermore, the minimal negative semidefinite solution of DARE~\eqref{dare} can be obtained by  $X_{-,m}=-Z_\infty^{-1}$.
\end{theorem}
%

Finally, the relationship between the almost stabilizabing solution and positive semidefinite extremal solutions of DARE \eqref{dare-b}
is summarized in the following proposition without proof.
\begin{proposition}
If $X_s$ is the unique almost stabilizabing solution of the DARE \eqref{dare-b}, then
\begin{enumerate}
 \item[(i)] $X_{+,M}$ exists and $X_s = X_{+,M} \geq 0$.
  \item[(ii)] $X_{+,m}$ exists. Moreover, $X_{+,M}=X_{+,m}$ if $\rho(T_{X_{+,m}})\leq 1$, and $X_{+,m}\neq X_{+,M}$ if $\rho(T_{X_{+,m}})>1$.
 \end{enumerate}
\end{proposition}

\section{Acceleration of fixed-point iteration} \label{sec4}
In previous sections we have addressed the convergence of fixed-point iteration
$ X_{k+1} = \mathcal{R}(X_k)$, $k \geq 0$,
for the existence of the extremal solutions to DARE \eqref{dare-b}, in which $X_0\in \mathbb{H}_n$ is some appropriate initial guess.
Since the FPI is usually linearly convergent, a numerical method of higher order of convergence is always required in the practical computation
and many real-life applications.
In this section, for any positive integer $r > 1$, we will revisit an accelerated FPI (AFPI) of the form
\begin{subequations} \label{afpi-Xk}
\begin{align}
  \widehat{X}_{k+1} &= \mathcal{R}^{(r^{k+1} - r^{k})}(\widehat{X}_k),\quad k\geq 1,\\
  \widehat{X}_{1} &=\mathcal{R}^{(r)}(\widehat{X}_0),\quad k=1
\end{align}
\end{subequations}
 with $\widehat{X}_{0}=X_0$, for solving the numerical solutions of DARE \eqref{dare-b}.
Here $\mathcal{R}^{(\ell)}(\cdot)$ denotes the composition of the operator $\mathcal{R}(\cdot)$ itself for $\ell$ times, where $\ell \geq 1$ is a positive integer.
Theoretically, the iteration of the form \eqref{afpi-Xk} is equivalent to the formula
\begin{align} \label{fpi-Zk}
    \widehat{X}_{k} &= \mathcal{R}^{(r^k)}(\widehat{X}_0),\quad k\geq 1,
\end{align}
with $\widehat{X}_0 = X_0$, and we see that
\begin{align}\label{acc}
 \widehat{X}_k = X_{r^{k}}
\end{align}
for each $k\geq 1$. Therefore, it is an interesting and chanllenging issue to make sure that the explicit expression of $\widehat{X}_k$ is of the form~\eqref{acc}
for each $k\geq 1$.
In the following sections we will show that the AFPI algorithm is more efficient than the iteration \eqref{fpi-Zk},
or even some Newton-type methods \cite{g98,z.c.l17}, for solving the extremal solutions of the DARE \eqref{dare-b}.

\subsection{Equivalent formulation of the fixed-point iteration} \label{sec4p1}
The following definition {\color{black} modifies} the semigroup property of the iteration associated a binary operator.
\begin{definition}~\cite{l.c20}\label{def1}
Let $\mathbb{K}_n \subseteq \mathbb{C}^{p\times q}$ and $F:\mathbb{K}_n \times \mathbb{K}_n \rightarrow \mathbb{K}_n$ be a binary matrix operator,
where $p$ and $q$ are positive integers.
We call that an iteration
\begin{equation*}
\mathbb{X}_{k+1}=F(\mathbb{X}_k,\mathbb{X}_0), \quad k\geq 0,
\end{equation*}
has the semigroup property if the operator $F$ satisfies the following associative rule:
\begin{align*}
F(F(Y,Z),W)=F(Y,F(Z,W))
\end{align*}
for any $Y,Z$ and $W$ in $\mathbb{K}_n$.
\end{definition}
It is interesting to point out that the sequence $\{{\mathbb{X}}_{k}\}$ satisfies the so-called discrete flow property~\cite{l.c20}, that is,
\begin{align}\label{DF}
{\mathbb{X}}_{i+j+1}=F({\mathbb{X}}_{i},{\mathbb{X}}_{j})
\end{align}
for any nonnegative integers $i$ and $j$. {\color{black} Here the subscript of \eqref{DF} is an equivalent adjustment to the original formula~\cite[Theorem 3.2]{l.c20}.}

For the FPI defined by $X_{k+1} = \mathcal{R}(X_k)$, in which $X_0\in \mathbb{H}_n$ is an initial matrix and
the operator $\mathcal{R}(\cdot)$ is defined by \eqref{dare-b},
it is shown in \cite{l.c18} that the fixed-point iteration can be rewritten as the following formulation
 \begin{equation}\label{fpi-eqform}
 X_{k+1} = \mathcal{R}^{(k)}(\mathcal{R}(X_0))= \mathcal{R}^{(k+1)}(X_0) = H_{k}+A_{k}^H X_0 (I+ G_{k} X_0)^{-1} A_{k},
 \end{equation}
where the sequence of matrices $\{ (A_k, G_k, H_k)\}_{k=0}^\infty$ is generated by the following iteration
\begin{equation} \label{opF}
\mathbb{X}_{k+1} = F(\mathbb{X}_k, \mathbb{X}_0) :=
\begin{bmatrix} A_0\Delta_{G_{k},{H_0}}A_{k}\\G_0+A_0\Delta_{G_{k},{H_0}}G_{k}A_0^H\\H_{k}+A_{k}^H H_0\Delta_{G_{k},{H_0}}A_{k}\end{bmatrix},
\end{equation}
with $\mathbb{X}_k := \left[ A_k^H\  G_k \ H_k \right]^H$ and $\mathbb{X}_0:= \left[ A^H\  G\ H \right]^H$ for each $k\geq 0$,
provided that the matrices $\Delta_{G_i, H_j} := (I+G_i H_j)^{-1}$ exists for all $i,j\geq 0$.
Notice that $F: \mathbb{K}_n \times \mathbb{K}_n \rightarrow \mathbb{K}_n$ is a binary operator with
$\mathbb{K}_n:= \mathbb{C}^{n\times n}\times \mathbb{H}_n \times \mathbb{H}_n$.
Moreover, it has been shown in Theorem 4.2 of \cite{l.c20} that the iteration \eqref{opF} has the semigroup property and thus satisfies \eqref{DF}.
Based on the equivalent expression of the FPI \eqref{fpi-eqform}, we shall develop an accelerated FPI for computing the extremal solutions to the DARE \eqref{dare-b}.
{\color{black}
\begin{remark}\label{ZG}
From the FPI~\eqref{fpi-nsd2} and \eqref{opF}, it is worth mentioning that $Z_k=G_k$ for each $k\geq 0$ when $Z_0=G_0=G$. Analogously, it follows from the discrete flow property~\eqref{DF} and iteration \eqref{opF} that $X_k=H_k$ for each $k\geq 0$ if $X_0=H_0=H$.
\end{remark}
}
 \subsection{The accelerated fixed-point iteration} \label{sec4p2}
In this section, we will mainly aim at some efficient ways for generating the sequence of matrices $\{(A_k, G_k, H_k)\}_{k=0}^\infty$
presented in the FPI \eqref{fpi-eqform}, where it may converge linearly.
 For the sake of simplicity, the operator $\mathbf{F}_\ell: \mathbb{K}_n \rightarrow \mathbb{K}_n$ is defined recursively by
 \begin{equation} \label{bfF}
   \mathbf{F}_{\ell+1} (\mathbb{X}) = F(\mathbb{X}, \mathbf{F}_\ell (\mathbb{X})),\quad \ell\geq 1,
 \end{equation}
 with $\mathbf{F}_1 (\mathbb{X}) = \mathbb{X}$ for all $\mathbb{X}\in \mathbb{K}_n$ and
$F(\cdot, \cdot)$ being defined by \eqref{opF}.
Furthermore, if we let $\mathbf{X}_k := \left[\mathbf{A}_k^H\ \mathbf{G}_k\ \mathbf{H}_k \right]^H\in \mathbb{K}_n$ for $k\geq 0$ and
 $\mathbf{X}_0 := \mathbb{X}_0$ be defined as in \eqref{opF}, then the sequence $\{\mathbf{A}_k, \mathbf{G}_k, \mathbf{H}_k \}$ generated by the iteration
 \begin{equation} \label{afpi-r2}
   \mathbf{X}_{k+1} = F(\mathbf{X}_k, \mathbf{X}_k) = \mathbf{F}_2 (\mathbf{X}_k),\quad k\geq 0,
 \end{equation}
 which is equivalent to the doubling or strctured doubling algorithms \cite{a77,c.f.l.w04}.
Indeed, according the semigroup and discrete flow property~\eqref{DF},
we have $\mathbf{A}_k = A_{2^{k}-1}$, $\mathbf{G}_k = G_{2^{k}-1}$ and $\mathbf{H}_k = H_{2^{k}-1}$ for each $k\geq 0$.
 That is, under the iteration \eqref{afpi-r2}, the sequence of matrices $\{(A_k, G_k, H_k)\}_{k=0}^\infty$ proceeds rapidly with their subscripts being the
 exponential numbers of base number $r=2$.
From the theoretical point of view, staring with a suitable matrix $X_0\in \mathbb{H}_n$, the sequence $\{X_k\}_{k=0}^\infty$ generated by \eqref{fpi-eqform} might also converge rapidly to its limit, if the limit exists.

Analogously, when the base number is $r=3$, it is suggested in \cite{l.c20} to consider an accelerated iteration of the form
 \begin{equation*} 
   \mathbf{X}_{k+1} = F(\mathbf{X}_k, F(\mathbf{X}_k, \mathbf{X}_k)) = \mathbf{F}_3 (\mathbf{X}_k),
   \quad k\geq 0,
 \end{equation*}
with $\mathbf{X}_0 = \mathbb{X}_0$, and we have $\{\mathbf{A}_k, \mathbf{G}_k, \mathbf{H}_k \}
= \{A_{3^{k}-1}, G_{3^{k}-1}, H_{3^{k}-1}\}$ for each $k\geq 0$.

Recently, for any positive integer $r>1$, Lin and Chiang proposed an efficient iterative method for generating
the sequence $\{(A_k,G_k,H_k)\}_{k=0}^\infty$ with order $r$ of R-convergence, provided that the operator $F(\cdot, \cdot)$ in \eqref{opF}
is well-defined, see, e.g., Algorithm 3.1 in \cite{l.c20}. Theoretically, this algorithm utilizes the following accelerated iteration
\begin{equation}\label{afpi-r}
    \mathbf{X}_{k+1} = \mathbf{F}_r (\mathbf{X}_k),\quad k\geq 0,
\end{equation}
with $\mathbf{X}_0:= \left[A^H\ G\ H \right]^H$ and $\mathbf{F}_r (\cdot)$ being the operator defined by \eqref{bfF},
for constructing $\mathbf{A}_k = A_{r^{k}-1}$, $\mathbf{G}_k = G_{r^{k}-1}$ and $\mathbf{H}_k = H_{r^{k}-1}$, respectively.
Therefore, combining \eqref{fpi-eqform} and \eqref{afpi-r}, we obtain the pseudocode of AFPI summarized in \cref{alg-afpi}.
\begin{algorithm}
  \caption{The Accelerated Fixed-Point Iteration with $r$ (AFPI($r$)) for solving
$$
X=H_0+A_0^H X (I_n+G_0 X)^{-1} A_0.
$$
}
  \label{alg-afpi}
\begin{algorithmic}
\Require $\mathbf{A}_0=A_0\in \mathbb{C}^{n\times n}$, $\mathbf{G}_0=B_0 R_0^{-1}B_0^H\geq 0$, $\mathbf{H}_0=H_0 \geq 0$, $(A_0,B_0)$ is stabilizable, initial matrix $\widehat{X}_0\in \mathbb{H}_n$ and $r > 1$.
\Ensure the maximal positive semidefinite solution $\widehat{X}_{k_1}$ and the minimal positive semidefinite solution $\mathbf{H}_{k_2}$.

\For{$k =0,1,2,\ldots$}
  \State $A_k^{(1)} = \mathbf{A}_k$,  $G_k^{(1)} = \mathbf{G}_k$,  $H_k^{(1)} = \mathbf{H}_k$;
  \For{$l = 1,2,\ldots,  r-2$}
    \State $A_k^{(l+1)} = A_k^{(l)} (I + \mathbf{G}_k H_k^{(l)})^{-1}\mathbf{A}_k$;
    \State $G_k^{(l+1)} = G_k^{(l)} + A_k^{(l)} (I + \mathbf{G}_k H_k^{(l)})^{-1}\mathbf{G}_k (A_k^{(l)})^H$;
    \State $H_k^{(l+1)} = \mathbf{H}_k + \mathbf{A}_k^H H_k^{(l)}  (I + \mathbf{G}_k H_k^{(l)})^{-1}\mathbf{A}_k$;
  \EndFor
  \State $\mathbf{A}_{k+1} = A_k^{(r-1)} (I + \mathbf{G}_k H_k^{(r-1)})^{-1}\mathbf{A}_k$;
  \State $\mathbf{G}_{k+1} = G_k^{(r-1)} + A_k^{(r-1)} (I + \mathbf{G}_k H_k^{(r-1)})^{-1}\mathbf{G}_k (A_k^{(r-1)})^H$;
  \State $\mathbf{H}_{k+1} = \mathbf{H}_k + \mathbf{A}_k^H H_k^{(r-1)} (I + \mathbf{G}_k H_k^{(r-1)})^{-1}\mathbf{A}_k$;
  \State $\widehat{X}_{k+1} = \mathbf{A}_{k+1}^H \widehat{X}_0 (I + \mathbf{G}_{k+1} \widehat{X}_0)^{-1}\mathbf{A}_{k+1} + \mathbf{H}_{k+1}$;
  \If{$\widehat{X}_{k_1}$ and $\mathbf{H}_{k_2}$ satisfy the stoping criterion for some positive  $k_1,k_2$}
    \State \Return the positive semidefinite matrices $\widehat{X}_{k_1}$ and $\mathbf{H}_{k_2}$;
  \EndIf
\EndFor
\end{algorithmic}
\end{algorithm}

Due to the property presented in~\eqref{acc}, it seems that the superlinear
convergence of the AFPI algorithm would be expected in practical computation. In our numerical experiments, we adopted $\mathbf{A}_0 = A$, $B_0 = B$, $R_0 = R$ and $\mathbf{H}_0 = H$ in \cref{alg-afpi}
for computing the positive semidefinite extremal solutions of the DARE \eqref{dare}. On the other hand,
the posive semidefinite extremal solutions of the dual DARE \eqref{dual} were computed if we adopted $\mathbf{A}_0 = \widehat{A}$, $B_0 = \widehat{B}$,
$R_0 = \widehat{R}$ and $\mathbf{H}_0 = \widehat{H}$ defined by \eqref{dual3}, respectively, as the initial data in \cref{alg-afpi}.

\subsection{Convergence analysis of the AFPI} \label{sec4p3}
As for \cref{alg-afpi}, worth mentioning is that it can be reduced to SDA iteration when $r=2$~\cite{c.f.l.w04}.
The following theorem guarantees, under the same sufficient conditions as in \cref{thm3p2} and the equivalent conditions presented in \cref{thm3p3},
the R-superlinear convergence of $\{\widehat{X}_k\}_{k=0}^\infty$ and $\{\mathbf{H}_k\}_{k=0}^\infty$ generated by \cref{alg-afpi}.
\begin{theorem}\label{thm4p2}
Let the sequences $\{\mathbf{H}_k\}_{k=0}^\infty$ and $\{\widehat{X}_k\}_{k=0}^\infty$
be generated by \cref{alg-afpi} with $\mathbf{A}_0 = A$, $\mathbf{G}_0 = G$ and $\mathbf{H}_0 = H$, respectively.
If the hypotheses of \cref{thm3p2} are satisfied, then
\begin{enumerate}
  \item[(i)] $\{\mathbf{H}_k\}_{k=0}^\infty$ converges at least R-superlinearly to $X_{+,m}$ with the rate of convergence
  \begin{equation*} 
   \limsup_{k\rightarrow \infty} \sqrt[r^k]{\|\mathbf{H}_k-X_{+,m}\|}\leq \rho_{\mathbb{D}} (T_{X_{+,m}})^2<1.
  \end{equation*}
  \item[(ii)] $\{\widehat{X}_k\}_{k=0}^\infty$ converges at least R-superlinearly to $X_{+,M}$ with the rate of convergence
\begin{equation*}  
 \limsup_{k\rightarrow \infty} \sqrt[r^k]{\|\widehat{X}_k-X_{+,M}\|}\leq \rho (T_{X_{+,M}})^2,
\end{equation*}
provided that $\rho (T_{X_{+,M}})<1$.
\end{enumerate}
\end{theorem}
\begin{proof}
From the AFPI \eqref{afpi-r} with $r>1$, the sequence $\{\mathbf{X}_{k}\}_{k=0}^\infty$ generated by \cref{alg-afpi} has the semigroup property
and thus satisfies the discrete flow property \eqref{DF}. With the construction of these two sequences, we see that
$\mathbf{H}_{k}={H}_{r^{k}-1}$ and $\widehat{X}_k = \mathcal{R}^{(r^{k}-1)} (X_0)$ for each $k\geq 0$,
see, e.g., Remark 4.1 of \cite{l.c18}.
The proof is straightforward from \cref{thm3p1} and \cref{thm3p2}.
\end{proof}

In \cref{thm4p2}, there is no further conclusion for the sequence $\{\mathbf{G}_k\}_{k=0}^\infty$ produced by \cref{alg-afpi}
under the stabilizability assumption. As mentioned in \cref{thm3p4} and \cref{ZG}, this sequence will also converge at least R-superlinearly under the controllability of the pair $(A,B)$,
which is stated without proof as follows.
\begin{corollary} \label{cor4p3}
Let $\{\mathbf{G}_k\}_{k=0}^\infty$, $\{\mathbf{H}_k\}_{k=0}^\infty$ and $\{\widehat{X}_k\}_{k=0}^\infty$ be sequences
generated by \cref{alg-afpi} with $\mathbf{A}_0 = A$, $\mathbf{G}_0 = G$ and $\mathbf{H}_0 = H$, respectively.
If $(A,B)$ is controllable and $D_\geq^{(2)}\cap\mathbb{N}_n\neq \emptyset$, then the following statements hold:
\begin{enumerate}
  \item[(i)]
   $\mathbf{G}_k=Z_{r^k-1}$ for $k\geq 0$, where the sequence $\{Z_k\}_{k=0}^\infty$ is generated by \eqref{fpi-nsd2} with $\mathbf{G}_0=Z_0=G$. Moreover, $\{\mathbf{G}_k\}_{k=0}^\infty$ converges at least R-superlinearly to the unique
  almost stabilizing solution $G_\infty > 0$ of the dual DARE \eqref{DY} with the rate of convergence
  \[  \limsup_{k\rightarrow \infty} \sqrt[r^k]{\|\mathbf{G}_k - G_\infty \|} \leq \rho_{\mathbb{D}} (T^{-1}_{G_\infty})^2< 1.  \]
  If, in addition, $A$ is nonsingular, then $X_{-,m}=-G_\infty^{-1} \leq 0$ is the unique almost antistabilizing solution to the DARE \eqref{dare-b}.
  %
  \item[(ii)] $\{\mathbf{H}_k\}_{k=0}^\infty$ and $\{\widehat{X}_k\}_{k=0}^\infty$ have the same results as \cref{thm4p2}.
\end{enumerate}
\end{corollary}

When $A$ is nonsingular, according to \cref{thm3p1-dual} , it is possible to characterize the R-superlinear convergence of \cref{alg-afpi},
with $\mathbf{A}_0 = \widehat{A}$, $\mathbf{G}_0 = \widehat{G}$ and $\mathbf{H}_0 = \widehat{H}$, for computing the negative semidefinite extremal
solutions of the DARE \eqref{dare-b}. It will be stated without proof as follows.
\begin{theorem}\label{thm4p4}
Let the sequences $\{\mathbf{H}_k\}_{k=0}^\infty$ and $\{\widehat{X}_k\}_{k=0}^\infty$ be generated by \cref{alg-afpi}
with $\mathbf{A}_0 = \widehat{A}$, $\mathbf{G}_0 = \widehat{G}$ and $\mathbf{H}_0 = \widehat{H}$
being defined by \eqref{dual3}, respectively. If the hypotheses of \cref{thm3p1-dual} are satisfied, then
\begin{enumerate}
  \item[(i)] $\{\mathbf{H}_k\}_{k=0}^\infty$ converges at least R-superlinearly to $-X_{-,M}$
  with the rate of convergence
  \begin{equation*} 
   \limsup_{k\rightarrow \infty} \sqrt[r^k]{\|\mathbf{H}_k+X_{-,M}\|}\leq \rho_{\mathbb{D}} (T_{X_{-,M}}^{-1})^2<1.
  \end{equation*}
  \item[(ii)] $\{\widehat{X}_k\}_{k=0}^\infty$ converges at least R-superlinearly to $-X_{-,m}$ with the rate of convergence
\begin{equation*}  
 \limsup_{k\rightarrow \infty} \sqrt[r^k]{\|\widehat{X}_k +{X}_{-,m}\|}\leq \rho (T_{{X}_{-,m}}^{-1})^2=\mu (T_{{X}_{-,m}})^{-2},
\end{equation*}
provided that $\rho (T_{{X}_{-,m}}^{-1})<1$ or $\mu (T_{{X}_{-,m}})>1$.
\end{enumerate}
\end{theorem}
\begin{remark}
\noindent
\begin{enumerate}
  \item[(i)] Under the hypotheses of \cref{cor4p3}, the sequences $\{G_k\}_{k=0}^\infty$ and $\{\widehat{X}_k\}_{k=0}^\infty$ obtained in \cref{thm4p4}
  might have different speed of convergence, even though the minimal solution $X_{-,m}$ can be computed from their limits.
  It can be seen that the convergence of $\{G_k\}_{k=0}^\infty$ is at least R-superlinear,
  but the convergence of $\{\widehat{X}_k\}_{k=0}^\infty$ is unclear if $\mu(T_{X_{-,m}})=1$.
  \item[(ii)] Although the statement (ii) in either \cref{thm4p2} or
\cref{thm4p4} is inconclusive when $\rho(T_{X_{+,M}})=1$ or $\mu(T_{X_{-,m}})=1$, it is observed from our numerical experiments that the sequence $\{\widehat{X}_k\}_{k=0}^\infty$ may converge R-linearly with at least
the rate $1/r$, which will be investigated in the future and report the theoretical results elsewhere.
\end{enumerate}
\end{remark}

\section{Numerical examples} \label{sec5}
In this section, we present four examples to illustrate the accuracy and efficiency of
the AFPI($r$) in \cref{alg-afpi} for solving the extremal solutions of the DARE \eqref{dare-b}. In the first three examples we compared the AFPI algorithm, through the sequence $\{\widehat{X}_k\}_{k=0}^\infty$ starting with some suitable initial $\widehat{X}_0$, with Newton's method (NTM) \cite{g98} for solving the maximal and (almost) stabilizing solution $X_{+,M}\geq 0$ of DARE \eqref{dare-b}. Here $\widehat{X}_0\geq 0$ is the unique solution of Stein matrix equation
\begin{equation} \label{sme-X0}
\mathcal{S}_{A_F}(X) := X - A_F^H X A_F = H+F^H R F,
\end{equation}
which can be computed by MATLAB command {\tt dlyap} directly, if $A_F = A-BF$ is d-stable for some $F\in \mathbb{C}^{m\times n}$.
Moreover, we assume $R=I$ in all numerical examples.

As a byproduct, the sequence $\{\mathbf{H}_k\}_{k=0}^\infty$ generated by \cref{alg-afpi} converges at least R-superlinearly to the minimal positive semidefinite
solution $X_{+,m}$ of the DAREs in \cref{ex1} and \cref{ex2}, which is shown theoretically in \cref{thm4p2}. Moreover, we also demonstrate the ability of
the AFPI algorithm for computing the negative semidefinite extremal solutions to the DARE in \cref{ex4}. Starting with $\widehat{A}$, $\widehat{G}$, $\widehat{H}$ and $\widehat{X}_0$ defined as in \eqref{dual3}, it is seen in \cref{thm4p4} that the AFPI produces two sequnces of matrices $\{-\widehat{X}_k\}_{k=0}^\infty$
and $\{-\mathbf{H}_k\}_{k=0}^\infty$ converging to the minimal negative semidefinite solution $X_{-,m}$ and
the maximal negative semidefinite solution $X_{-,M}$ of the DARE \eqref{dare-b}, respectively.

For an approximate solution $Z$ to the DARE~\eqref{dare},
 we will report its normalized residual
\[  NRes(Z) := \frac{\|Z - \mathcal{R}(Z)\|}{\|Z\| + \|A^H Z(I+GZ)^{-1}A\| + \|H\|},  \]
and one of the following two quantities associated with $T_Z:= (I+GZ)^{-1}A$
\[ \rho (T_Z):= \max\{|\lambda| \ |\ \lambda\in \sigma (T_Z)\},\quad
\mu (T_Z):= \min\{|\lambda| \ |\  \lambda\in \sigma (T_Z)\}   \]
in the tables below.
We terminated the numerical methods AFPI and NTM when $NRes\leq 1.0\times 10^{-15}$ in Examples \ref{ex1}--\ref{ex3}, and
the AFPI algorithm terminated when $NRes\leq 1.0\times 10^{-12}$ in \cref{ex4}, respectively.
All numerical experiments were performed on ASUS laptop (ROG GL502VS-0111E7700HQ),
using Microsoft Windows 10 operating system and MATLAB Version R2019b,
with Intel Core i7-7700HQ CPU and 32 GB RAM.
\begin{example} \label{ex1}
Let the coefficient matrices of DARE \eqref{dare-b} be given by
\[  A = \begin{bmatrix} 3 & 0\\ 0 & 1/2 \end{bmatrix},\quad
B = \begin{bmatrix} 1\\ 0 \end{bmatrix},\quad
C = \begin{bmatrix} 0 & 1 \end{bmatrix}. \]
Then it is easily seen that the pair $(A,B)$ is stabilizable, but $(A,C)$ is not detectable.
Moreover, this DARE \eqref{dare-b} has only two positive semidefinite solutions, namely,
\[  X_{+,M} = \begin{bmatrix} 8 & 0\\ 0 & 4/3 \end{bmatrix},\quad X_{+,m} = \begin{bmatrix} 0 & 0\\ 0 & 4/3 \end{bmatrix},  \]
with $H = C^H C$. The matrix $X_{+,M}$ is the maximal and stabilizing solution of the DARE such that the eigenvalues of
$T_{X_{+,M}} = (I+GX_{+,M})^{-1}A$ are $1/3$ and $1/2$, and $\sigma (T_{X_{+,m}}) = \sigma (A)$, respectively.
Thus, $X_{+,m}$ is the minimal positive semidefinite solution of the DARE \eqref{dare-b} with
the property $\rho (T_{X_{+,m}}) = 3 > 1$.

If we choose $F = [3, 0]$ so that $A_F = A-BF$ is d-stable and $\widehat{X}_0$ is determined by \eqref{sme-X0}, then the numerical results of AFPI(2) and NTM
are presented in \cref{tab-1} and \cref{tab-2}, respectively,
where $\mathbf{A}_0 = A$, $\mathbf{G}_0 = BB^H$ and $\mathbf{H}_0 = H$ are required matrices for \cref{alg-afpi}.
Notice that our algorithm AFPI(2) generates two highly accurate approximations to the extremal solutions $X_{+,M}$ and $X_{+,m}$ simultaneously,
with the relative errors being $2.3\times 10^{-16}$ and $0.0\times 10^0$ respectively,
even though $\| \mathbf{A}_k\|$ increases rapidly as $k$ proceeds.
Nevertheless, starting from the same initial matrix $X_0$, the NTM only produces an accurate approximation to the solution $X_{+,M}$.
\begin{table}[tbhp] 
\begin{center}
\begin{tabular}{c|cc|cc|c} 
  \hline
  $k$ & $NRes(\widehat{X}_k)$ & $NRes(\mathbf{H}_k)$ & $\rho (T_{\widehat{X}_k})$ & $\rho (T_{\mathbf{H}_k})$ & $\|\mathbf{A}_k\|$ \\ 
 \hline\hline
  $1$ & $5.7\times 10^{-4}$ & $2.4\times 10^{-2}$  & $5.0\times 10^{-1}$ & $3.0\times 10^0$  & $9.0\times 10^0$ \\
  $2$ & $7.1\times 10^{-6}$ & $1.5\times 10^{-3}$  & $5.0\times 10^{-1}$ & $3.0\times 10^0$ & $8.1\times 10^1$ \\
    $3$ & $1.1\times 10^{-9}$ & $5.7\times 10^{-6}$  & $5.0\times 10^{-1}$ & $3.0\times 10^0$ & $6.6\times 10^3$ \\
    $4$ & $1.0\times 10^{-16}$ & $8.7\times 10^{-11}$  & $5.0\times 10^{-1}$ & $3.0\times 10^0$ & $4.3\times 10^7$ \\
    $5$ &  & $0.0\times 10^{0}$  &  & $3.0\times 10^0$& $1.9\times 10^{15}$ \\
   \hline
 \end{tabular}
\end{center}
\caption{Numerical results of AFPI(2) for \cref{ex1}.}\label{tab-1}
\end{table}
\begin{table}[tbhp]
  \begin{center}
    \begin{tabular}{c|cc}
       \hline
       $k$ & $NRes(X_k)$ & $\rho (T_{X_k})$ \\
       \hline\hline
       $1$ & $5.7\times 10^{-4}$ & $5.0\times 10^{-1}$ \\
       $2$ & $8.7\times 10^{-8}$ & $5.0\times 10^{-1}$ \\
       $3$ & $2.1\times 10^{-15}$ & $5.0\times 10^{-1}$ \\
       $4$ & $0.0\times 10^{0}$ & $5.0\times 10^{-1}$ \\
       \hline
     \end{tabular}
  \end{center}
  \caption{Numerical results of NTM for \cref{ex1}.}\label{tab-2}
\end{table}
\end{example}

\begin{example} \label{ex2}
In this example we consider the DARE \eqref{dare-b}
with its $5\times 5$ coefficient matrices being defined by
  \[ A = \begin{bmatrix}
    2.9 & 1 & 0 & 0 & 0\\
    0 & 2.9 & 0 & 0 & 0 \\
    0 & 0 & 0 & 0 & 0 \\
    0 & 0 & 0 & 0 & 0 \\
    0 & 0 & 0 & 0 & 1
     \end{bmatrix},\quad
     H = \begin{bmatrix}
    0 & 0 & 0 & 0 & 0\\
    0 & 0 & 0 & 0 & 0 \\
    0 & 0 & 200 & -0.5 & 0 \\
    0 & 0 & -0.5 & 200 & 0 \\
    0 & 0 & 0 & 0 & 1
     \end{bmatrix}   \]
  and $B = \mathrm{diag}(\sqrt{2}, 1, 0, 0, 1)$, respectively.
It can be shown that the minimal positive semidefinite solution $X_{+,m}$ of the DARE \eqref{dare-b} is
\[ X_{+,m} = \begin{bmatrix}
    0 & 0 & 0 & 0 & 0\\
    0 & 0 & 0 & 0 & 0 \\
    0 & 0 & 200 & -0.5 & 0 \\
    0 & 0 & -0.5 & 200 & 0 \\
    0 & 0 & 0 & 0 & (1+\sqrt{5})/2
     \end{bmatrix},   \]
which is almost the same as $H$ except the $(5,5)$-entry.
 If we let $F = \mathrm{diag}(2, 3, 0, 0, 0.5)$ so that $\rho (A_F) < 1$ and $\widehat{X}_0$ is determined by \eqref{sme-X0},
 the numerical results of AFPI(2) and NTM are reported in \cref{tab1-ex2} and \cref{tab2-ex2}, respectively, for solving the DARE \eqref{dare-b} from the same initial $\widehat{X}_0\geq 0$.
\begin{table}[tbhp] 
\begin{center}
\begin{tabular}{c|cc|cc|c}
  \hline
  $k$ & $NRes(\widehat{X}_k)$ & $NRes(\mathbf{H}_k)$ & $\rho (T_{\widehat{X}_k})$ & $\rho (T_{\mathbf{H}_k})$ & $\|\mathbf{A}_k\|$ \\
 \hline\hline
  $1$ & $9.0\times 10^{-5}$ & $2.5\times 10^{-4}$  & $3.8\times 10^{-1}$ & $2.9\times 10^0$  & $1.2\times 10^1$ \\
  $2$ & $1.7\times 10^{-6}$ & $5.6\times 10^{-6}$  & $3.8\times 10^{-1}$ & $2.9\times 10^0$ & $1.3\times 10^2$ \\
    $3$ & $7.1\times 10^{-10}$ & $2.6\times 10^{-9}$  & $3.8\times 10^{-1}$ & $2.9\times 10^0$ & $1.5\times 10^4$ \\
    $4$ & $7.2\times 10^{-17}$ & $5.2\times 10^{-16}$  & $3.8\times 10^{-1}$ & $2.9\times 10^0$ & $1.4\times 10^8$ \\
    \hline
 \end{tabular}
\end{center}
\caption{Numerical results of AFPI(2) for \cref{ex2}.}\label{tab1-ex2}
\end{table}
\begin{table}[tbhp]
  \begin{center}
    \begin{tabular}{c|cc}
       \hline
       $k$ & $NRes(X_k)$ & $\rho (T_{X_k})$ \\
       \hline\hline
       $1$ & $3.4\times 10^{-4}$ & $3.8\times 10^{-1}$ \\
       $2$ & $1.3\times 10^{-6}$ & $3.8\times 10^{-1}$ \\
       $3$ & $2.6\times 10^{-11}$ & $3.8\times 10^{-1}$ \\
       $4$ & $2.3\times 10^{-18}$ & $3.8\times 10^{-1}$ \\
       \hline
     \end{tabular}
  \end{center}
  \caption{Numerical results of NTM for \cref{ex2}.} \label{tab2-ex2}
\end{table}

 Notice that, after $4$ iterations, the AFPI(2) generates accurate approximations to the
 maximal and minimal positive semidefinite solutions of the DARE \eqref{dare-b} simultaneously, but NTM merely computes
 an approximation to the solution $X_{+,M}$. Since the exact maximal solution is unknown, we may
 apply the MATLAB command {\tt dare} to produce the accurate approximation of $X_{+,M}$. The relative errors of the $4$th iterates computed by AFPI(2) are
 \[  \frac{\|\widehat{X}_4 - X_{+,M}\|}{\|X_{+,M}\|} \approx 1.6\times 10^{-16},\quad
 \frac{\|\mathbf{H}_4 - X_{+,m}\|}{\|X_{+,m}\|}\approx 1.2\times 10^{-15},  \]
 respectively, which shows the feasibility of our proposed algorithm.
 \end{example}

 \begin{example}\label{ex3}
 This example is modified from Example 6.2 of \cite{g98}.
For $\varepsilon \geq 0$, the coefficient matrices of DARE \eqref{dare-b} are defined by
\begin{align*}
  A &= \mathrm{diag}\left( \begin{bmatrix}
    -1 & 0 & 0\\ 0 & 1 & \varepsilon \\ 0 & 0 & 1
  \end{bmatrix}, \begin{bmatrix} \frac{\sqrt{3}}{2} & \frac{1}{2} \\ \frac{-1}{2} & \frac{\sqrt{3}}{2} \end{bmatrix},
  \begin{bmatrix}
    \frac{1}{2} & 1 & 0\\ 0 & \frac{1}{2} & 1\\ 0 & 0 & \frac{1}{2}
  \end{bmatrix}    \right), \\
  B &= \begin{bmatrix}
    1 & 0 & 0 & 0 & 0 & 0 & 0 & 0\\
    1& 1 & \varepsilon & 0 & 0 & 0 & 0 & 0 \\
    0 & 1 & 1 & 0 & 0 & 0 & 0 & 0 \\
    0 & 0 & 1 & 1 & 0 & 0 & 0 & 0 \\
    0 & 0 & 0 & 1 & 1 & 0 & 0 & 0 \\
    0 & 0 & 0 & 0 & 1 & 1 & 0 & 0 \\
    0 & 0 & 0 & 0 & 0 & 1 & 1 & 0 \\
    0 & 0 & 0 & 0 & 0 & 0 & 1 & 1
  \end{bmatrix},\quad H = C^H C = 0\in \mathbb{R}^{8\times 8}.
\end{align*}
Then this DARE has a unique positive semidefinite solution $X_{+,M} = X_{+,m} = 0$, where $X_{+,M}$ is the almost stabilizing solution
with $\sigma(T_{X_{+,M}}) = \sigma (A)$ for all $\varepsilon\geq 0$.
Note that this DARE \eqref{dare-b} is just the same as the one appeared in Example 6.2 of \cite{g98} when $\varepsilon = 0$,
in which all unimodular eigenvalues of $A$ are semisimple.

If we choose $F = \mathrm{diag} (-1,1,1,1,1,0.1,0.1,0.1)$ for any $\varepsilon \geq 0$, then $A_F = A-BF$ is d-stable so that the initial matrix $\widehat{X}_0$ is the
unique positive semidefinite solution to the Stein equation \eqref{sme-X0}.
With $\mathbf{H}_0 = H = 0$, the AFPI($r$) generates $\mathbf{H}_k = 0$ for all $k\geq 0$ and $r\geq 2$.
When $\varepsilon = 0$, the convergence histories of NTM and AFPI($r$) are presented in \cref{fig-ex3}
for $r = 2,4, 8$ and $100$, respectively.
Moreover, the CPU times of these numerical methods are reported in \cref{tab1-ex3}.
\begin{table}[tbhp]
  \begin{center}
    \begin{tabular}{c|cc}
      \hline
      Method & Iter. No. & CPU Time (sec.) \\
      \hline\hline
      AFPI(2) & $50$ & $8.78\times 10^{-3}$ \\
      AFPI(4) & $25$ & $1.45\times 10^{-2}$   \\
      AFPI(8) & $17$ & $1.35\times 10^{-2}$  \\
       AFPI(100) & $8$ & $1.51\times 10^{-2}$  \\
        NTM & $50$ & $3.08\times 10^{-2}$  \\
      \hline
    \end{tabular}
    \caption{The CPU times of numerical methods for \cref{ex3} with $\varepsilon = 0$.} \label{tab1-ex3}
  \end{center}
\end{table}
\begin{figure}[tbhp]
\centering
\includegraphics[height=7cm,width=9cm]{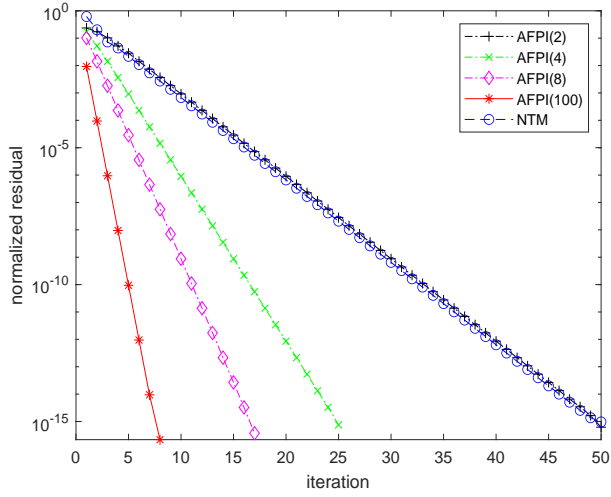}
\caption{Convergence histories of numerical methods for \cref{ex3} with $\varepsilon = 0$.}
\label{fig-ex3}
\end{figure}

As for $\varepsilon = 1$, the matrix $A$ contains an $2\times 2$ Jordan block corresponding to the eigenvalue $\lambda = 1$,
and the other unimodular eigenvalues are semisimple. In this case, the NTM encountered a breakdown after $54$ iterations, due to the stagnation of {\tt dlyap} for solving the corresponding Stein equation. It requires about $3.16\times 10^{-2}$ second to achieve the absolute error
$\|X_k - X_{+,M}\| \approx 5.4\times 10^{-9}$. On the other hand, the AFPI(100) produced a highly accurate approximation to the solution $X_{+,M} = 0$
with $\|\widehat{X}_k - X_{+,M}\| \approx 2.0\times 10^{-14}$ after $7$ iterations, starting from the same initial matrix $\widehat{X}_0$ determined as NTM by \eqref{sme-X0}.
Moreover, the numerical results of AFPI(100) are reported in \cref{tab2-ex3}. From the third column of \cref{tab2-ex3} it seems that
the AFPI(100) converges linearly with rate of convergence being $\frac{1}{r} = 1.00\times 10^{-2}$ approximately, but
this phenomenon will be further investigated in the future and reported elsewhere.
\begin{table}[tbhp] 
  \begin{center}
    \begin{tabular}{c|cc|cc}
      \hline
      $k$ & $\|\widehat{X}_k - X_{+,M}\|$ & $\frac{\|\widehat{X}_k - X_{+,M}\|}{\|\widehat{X}_{k-1} - X_{+,M}\|}$ & $\rho (T_{\widehat{X}_k})$ & $\|\mathbf{A}_k\|$ \\
      \hline\hline
      $1$ & $2.0\times 10^{-2}$ & $1.06\times 10^{-3}$ & $9.90\times 10^{-1}$ & $1.0\times 10^{2}$ \\
      $2$ & $2.0\times 10^{-4}$ & $1.02\times 10^{-2}$ & $1.0\times 10^{0}$ & $1.0\times 10^{4}$ \\
      $3$ & $2.0\times 10^{-6}$ & $1.00\times 10^{-2}$ & $1.0\times 10^{0}$ & $1.0\times 10^{6}$ \\
      $4$ & $2.0\times 10^{-8}$ & $1.00\times 10^{-2}$ & $1.0\times 10^{0}$ & $1.0\times 10^{8}$ \\
      $5$ & $2.0\times 10^{-10}$ & $1.00\times 10^{-2}$ & $1.0\times 10^{0}$ & $1.0\times 10^{10}$ \\
      $6$ & $2.0\times 10^{-12}$ & $1.00\times 10^{-2}$ & $1.0\times 10^{0}$ & $1.0\times 10^{12}$ \\
      $7$ & $2.0\times 10^{-14}$ & $1.00\times 10^{-2}$ & $1.0\times 10^{0}$ & $1.0\times 10^{14}$ \\
        \hline
    \end{tabular}
    \caption{Numerical results of AFPI(100) for \cref{ex3} with $\varepsilon = 1$.} \label{tab2-ex3}
    \end{center}
    \end{table}
\end{example}

\begin{example} \label{ex4}
This example will demonstrate the feasibility of our AFPI algorithm for solving
the negative semidefinite extremal solutions of the DARE \eqref{dare-b}. As quoted from Example 6.2 of \cite{z.c.l17}, the coefficient matrices
of DARE \eqref{dare-b} are given by
\[   A = \begin{bmatrix}
  4 & 3\\ \frac{-9}{2} & \frac{-7}{2}
\end{bmatrix},\quad B = \begin{bmatrix}
  6\\ -5
\end{bmatrix},\quad H = \begin{bmatrix}
  9 & 6\\ 6 & 4
\end{bmatrix}.  \]
Moreover, the authors in \cite{z.c.l17} suggested that the given DARE \eqref{dare-b}
should be transformed into its dual DARE \eqref{ddare}
for computing the negative semidefinite extremal solutions of the DARE \eqref{dare-b}, where
 \[ \widetilde{A} = \begin{bmatrix} 7 & 6\\ -9 & -8 \end{bmatrix},\quad
  \widetilde{B} = \begin{bmatrix} 12\\ -14 \end{bmatrix},\quad  \widetilde{C} = \begin{bmatrix} 24 & 16 \end{bmatrix},\quad
   \widetilde{R} = 65,\quad \widetilde{H} = H.    \]
 It is shown in \cite{z.c.l17} that the given DARE \eqref{dare-b} has three extremal solutions, namely,
\[  
X_{+,M} = X_{+,m} = \begin{bmatrix}
  \frac{9}{2}+\frac{9}{8}\sqrt{17} & 3+\frac{3}{4}\sqrt{17} \\[1ex]
  3+\frac{3}{4}\sqrt{17} & 2+\frac{\sqrt{17}}{2}
\end{bmatrix} \]
and
\[ X_{-,M} = \begin{bmatrix}
  \frac{9}{2}-\frac{9}{8}\sqrt{17} & 3-\frac{3}{4}\sqrt{17} \\[1ex]
  3-\frac{3}{4}\sqrt{17} & 2-\frac{\sqrt{17}}{2}
\end{bmatrix},\quad X_{-,m} = \begin{bmatrix}
  \frac{-103}{12}-\frac{\sqrt{17}}{8} & \frac{-39}{4}-\frac{\sqrt{17}}{4} \\[1ex]
   \frac{-39}{4}-\frac{\sqrt{17}}{4} & \frac{-43}{4}-\frac{\sqrt{17}}{2}
\end{bmatrix},   \]
where $X_{+,M}\geq 0$ is the maximal and stabilizing solution, $X_{-,M}$ is the maximal negative semidefinite solution and
$X_{-,m}\leq 0$ is the minimal sesolution, respectively.

 Since it can be checked that $90I \in \mathcal{D}_\geq^{(2)}\cap\mathbb{N}_n$ and $(A,B)$ is controllable, the sufficient conditions of \cref{cor4p3} are satisfied.
If we let $F = [-0.58, -0.68]$ so that $A_F$ is d-stable, then the AFPI(2), with $\mathbf{A}_0 = A$, $\mathbf{G}_0 = G$, $\mathbf{H}_0 = H$ and
$\widehat{X}_0$ being determined by \eqref{sme-X0}, produced highly accurate approximations to $X_{+,M}$ and $X_{-,m}$
with relative errors being $1.8\times 10^{-16}$ and $5.9\times 10^{-14}$, respectively, after $5$ iterations.

In addition, for computing the negative semidefinite extremal solutions to the DARE \eqref{dare-b},
we consider the dual DARE \eqref{dual}
with $\widehat{B} = \widetilde{B}$, $\widehat{R} = \widetilde{R}$, $\widehat{A} = \widetilde{A} - \widehat{B}\widehat{R}^{-1}\widetilde{C}$,
$\widehat{G} = \widehat{B}\widehat{R}^{-1}\widehat{B}^H$ and $\widehat{H} = \widetilde{H} - \widetilde{C}^H \widehat{R}^{-1}\widetilde{C}$, respectively.
Since the pair $(\widehat{A}, \widehat{B})$ is stabilizable, we may choose $\widehat{F} = [0.62, 0.52]$ so that
the matrix $\widehat{A}_{\widehat{F}} := \widehat{A} - \widehat{B}\widehat{F}$ is d-stable.
Let $\widehat{X}_0\geq 0$ be the unique solution to the Stein equation
\begin{equation*} 
  \mathcal{S}_{\widehat{A}_{\widehat{F}}}(X) = \widehat{H} + \widehat{F}^H \widehat{R} \widehat{F} \geq 0.
 \end{equation*}
Applying \cref{alg-afpi} with $\mathbf{A}_0 = \widehat{A}$, $\mathbf{G}_0 = \widehat{G}$, $\mathbf{H}_0 = \widehat{H}$
and $\widehat{X}_0$, the AFPI(4) generates highly accurate approximations $\mathbf{H}_3\approx -X_{-,M}$ and
$\widehat{X}_3\approx -X_{-,m}$ with relative errors being $4.1\times 10^{-14}$ and $7.4\times 10^{-16}$, respectively, after $3$ iterations.
Moreover, the numerical results of AFPI(4) are reported in \cref{tab1-ex4}. From the $5$th column of \cref{tab1-ex4}, we see that
our AFPI can provide a good approximation to the minimal and antistabilizing solution $X_{-,m}$ of the DARE \eqref{dare}.
\begin{table}[tbhp] 
\begin{center}
\begin{tabular}{c|cc|cc|c}
  \hline
  $k$ & $NRes(-\mathbf{H}_k)$ & $NRes(-\widehat{X}_k)$ & $\mu (T_{-\mathbf{H}_k})$ & $\mu (T_{-\widehat{X}_k})$ & $\|\mathbf{A}_k\|$ \\
 \hline\hline
  $1$ & $2.1\times 10^{-13}$ & $9.5\times 10^{-1}$  & $5.0\times 10^{-1}$ & $2.0\times 10^0$  & $5.7\times 10^1$ \\
  $2$ & $2.1\times 10^{-13}$ & $5.1\times 10^{-8}$  & $5.0\times 10^{-1}$ & $2.0\times 10^0$ & $2.3\times 10^5$ \\
    $3$ & $2.1\times 10^{-13}$ & $7.4\times 10^{-13}$  & $5.0\times 10^{-1}$ & $2.0\times 10^0$ & $1.3\times 10^{-1}$ \\
       \hline
 \end{tabular}
\end{center}
\caption{Numerical results of AFPI(4) for \cref{ex4}.}\label{tab1-ex4}
\end{table}
\end{example}

\section{Concluding remarks}\label{sec6}
 In most of the past works, it is always assumed that the DARE has a unique maximal positive semidefinite solution $X$ with $\rho(T_X)\leq 1$ and another meaningful solutions are lacking in brief discussion. Our contribution fills in the existing gap in finding four extremal solutions of the DARE. More precisely, we have studied the existence of minimal and maximum solutions under mild assumptions. It is important to note that we no longer need the traditional assumptions on the distribution of the spectrum of the symplectic matrix pencil or Popov matrix function, but give mathematical conditions from different point of view.

This paper concerns comprehensive convergence analysis of the most recent advanced algorithms AFPI, including its variant SDA, for solving the extremal solutions of DARE. We verify that the AFPI works efficiently and the convergence speed is R-superlinear under the mild assumptions. As compared to the previous works, the theoretical results presented here are merely deduced from the framework of the fixed-point iteration, using basic assumptions and elementary matrix theory. We believe the results we obtain are novel on this topic and could provide considerable insights into the study of unmixed solutions of DARE.


\end{document}